\documentclass[reqno]{amsart}
\usepackage[dvipsnames]{xcolor}

\usepackage[numbers]{natbib}  %

\usepackage{enumerate}

\newcommand{\bbC}{\C}

\usepackage{ts_def_share}

\usepackage[colorlinks,urlcolor=red,citecolor=blue,linkcolor=red]{hyperref}
 
 \newcommand{\la}{\langle}
 \newcommand{\ra}{\rangle}

\renewcommand{\bH}{H}

\newcommand{\bbr}{\mathbb{R}}
\newcommand{\bbe}{\mathbb{E}}
\newcommand{\bbn}{\mathbb{N}}
\newcommand{\bbp}{\mathbb{P}}
\newcommand{\bbq}{\mathbb{Q}}

\newcommand{\bbb}{\mathbb{B}}

\newcommand{\calf}{\mathscr{F}}

\newcommand{\calx}{\mathcal{X}}

\newcommand{\Id}{{\rm Id}}

\newcommand{\lin}{{\rm lin}}
\newcommand{\ran}{{\rm ran}}

\definecolor{orange}{RGB}{255,140,0}

\begin{document}

\title{{\sf \large Infinite dimensional affine processes }}
    \author {Thorsten Schmidt}
    \address{Freiburg Institute of Advanced Studies (FRIAS), Germany. 
 University of Strasbourg Institute for Advanced Study (USIAS), France. 
 University of Freiburg, Department of Mathematical Stochastics, Ernst-Zermelo-Str. 1, 79104 Freiburg, Germany. }
 
    \email{thorsten.schmidt@stochastik.uni-freiburg.de}
    \author{Stefan Tappe}
       \address{Karlsruhe Institute of Technology, Institute of Stochastics, Postfach 6980, 76049 Karlsruhe, Germany. }
    \email{jens-stefan.tappe@kit.edu}
    \author{Weijun Yu}
    \address{d-fine GmbH, An der Hauptwache 7, 60313 Frankfurt am Main, Germany. }
    \email{weijun.yu@d-fine.de}
    \thanks{We are grateful to Christa Cuchiero, Philipp Harms and Josef Teichmann for valuable discussions. We are also grateful to the editor and two referees for helpful comments and suggestions. Financial support from the DFG in project number 196379142 and from the Freiburg Institute of Advanced Studies (FRIAS) is gratefully acknowledged.}
    \date{\today}
\maketitle

\begin{abstract}
The goal of this article is to investigate infinite dimensional affine diffusion processes on the canonical state space. This includes a derivation of the corresponding system of Riccati differential equations and an existence proof for such processes, which has been missing in the literature so far. For the existence proof, we will regard affine processes as solutions to infinite dimensional stochastic differential equations with values in Hilbert spaces. This requires a suitable version of the Yamada-Watanabe theorem, which we will provide in this paper. Several examples of infinite dimensional affine processes accompany our results.
\end{abstract}

\vspace{2mm}

\section{Introduction}

Affine processes constitute an important model class due to their analytical tractability; in particular regarding applications in the field of mathematical finance. There is a substantial literature about affine processes in finite dimension. We refer, for example, to \cite{DuffieFilipovicSchachermayer, Filipovic05, Filipovic-Mayerhofer, KellerResselSchachermayerTeichmann2011, GabrielliTeichmann2018} for affine processes on the canonical state space and to \cite{CFMT2010, SpreijVeerman, KellerRessel2013, CuchieroTeichmann2013, KR-Mayerhofer, CKMT2016} for affine processes on more general state spaces. Some recent and related developments are affine processes with stochastic discontinuities (see \cite{KRSchmidtWardenga}), affine processes under parameter uncertainty (see \cite{FadinaNeufeldSchmidt}) and polynomial processes (see \cite{CuchieroKellerresselTeichmann2012, FilipovicLarsson2016, Cuchiero2018, CuchieroLarssonSvaluto2018b}).

Only recently, increasing interest evolved in infinite dimensional affine processes: the theory of probability measure-valued processes has been utilized in \cite{CuchieroLarssonSvaluto2018} for the study of polynomial diffusions. We also mention the works \cite{Handa}, \cite{Tomczyk} and \cite{Hambly}, where some examples, such as infinite dimensional square-root processes and infinite dimensional Heston type processes, are treated within the framework of probability measure-valued stochastic processes. Another recent approach to polynomial processes in infinite dimension is the paper \cite{BenthDeteringKruehner2018}, where the notion of a polynomial process -- in the sense that polynomials are preserved under conditional expectations -- is extended to a Banach space.

The general study in \cite{Grafendorfer} deals with affine processes in infinite dimension on general state spaces; more precisely affine processes are understood as processes with an {exponential} affine structure of the characteristic exponent, and they are studied on topological vector spaces which do not need to be separable or metrizable. The work \cite{YuWeijun}, of which the present paper constitutes a further development in certain aspects, studies the special case of affine processes with values in separable Hilbert spaces; with a special focus to applications in finance. Some recent articles deal with particular examples of affine processes with values in Hilbert spaces, such as Ornstein-Uhlenbeck processes with stochastic volatility and tensor Heston type processes; see, for example \cite{BenthRuedigerSuess2018} and \cite{BenthSimonsen2018}.

The paper \cite{Tappe-affin} is between the finite and the infinite dimensional setting. More precisely, therein it has been investigated when the solutions to a (infinite dimensional) stochastic partial differential equation admit a finite dimensional realization with (finite dimensional) affine state processes.

The goal of the present paper is to explore infinite dimensional affine diffusion processes on the canonical state space. This includes a derivation of the corresponding system of Riccati differential equations and an existence proof for infinite dimensional affine processes, which has been missing in the literature so far. For the existence proof, we regard affine processes as solutions to infinite dimensional stochastic differential equations (SDEs) with values in Hilbert spaces. This requires a suitable version of the Yamada-Watanabe theorem, and -- in order to apply the Yamada-Watanabe theorem -- sufficient conditions for the existence of weak solutions and for pathwise uniqueness of solutions. Infinite dimensional versions of the Yamada-Watanabe theorem can be found in \cite{Ondrejat}, \cite{Roeckner} and \cite{Tappe-YW}. However, none of these results can directly be applied in our setting, and for this reason we provide a self-contained version in this paper. In order to ensure the existence of weak solutions, we establish a refined version of a result from \cite{GMR}, where the main idea is to consider starting points from an appropriate retracted subspace with compact embedding, and for this reason we need a suitably adjusted version of the Yamada-Watanabe theorem. The pathwise uniqueness follows from a version of the uniqueness result from \cite{Yamada-Watanabe-1971} in infinite dimension.

The remainder of this paper is organized as follows. In Section \ref{sec-affine} we introduce affine processes and derive a general Riccati system for the functions appearing in the characteristic exponent. In Section \ref{sec-existence} we provide the existence result for affine processes in the spirit of strong solutions to infinite dimensional SDEs. In Section \ref{sec-examples} we present examples, where our existence result applies; this includes infinite dimensional processes of Cox-Ingersoll-Ross type and infinite dimensional processes of Heston type. For convenience of the reader, the proof of Lemma \ref{lemma 2.8} is deferred to Appendix \ref{sec-aux}. Moreover, Appendix \ref{app-SDE} contains the required results about SDEs in Hilbert spaces; in particular the adjusted version of the Yamada-Watanabe theorem, and the mentioned results about existence of weak solutions and pathwise uniqueness. Finally, in Appendix \ref{app-linear-operators} we provide the required results about linear operators in Hilbert spaces.

\section{Infinite dimensional affine processes}\label{sec-affine}

Affine models and their applications to dynamic term structure modelling have been intensively studied, mostly focusing on finite-dimensional affine models where the dimension, or the number of factors, is known and fixed. Here, we do not restrict the number of factors to be known or finite but rather study affine processes from an infinite-dimensional perspective. For practical applications, this allows to treat the number of factors as unknown parameter which has to be estimated. For the construction of infinite dimensional affine processes we follow the approaches in \cite{DuffieFilipovicSchachermayer,KellerResselSchachermayerTeichmann2011}. The used techniques for Hilbert-space valued stochastic analysis is taken from \cite{DaPratoZabczyk}.

Let $(\bH,\langle \cdot, \cdot \rangle)$ be an infinite-dimensional and separable Hilbert space with scalar product $\langle \cdot, \cdot \rangle$ and associated norm $\parallel \cdot \parallel$. The adjoint of a linear operator $T\in L(H)$ is denoted by $T^*$. By $\ccB(\bH)$ we denote the associated Borel $\sigma$-algebra. We fix throughout an orthonormal basis $(e_i)_{i=1}^\infty$ of $\bH$.  %

Affine processes are characterized by  the convenient property that their Fourier transforms have exponential affine form. For the study of Fourier transform we introduce the following \emph{complexification} of $\bH$: set
$$ \bH_{\bbC} = \{x+iy: x,y \in \bH\} $$
and equip it with the inner product $\langle x+iy, u+iv \rangle_{H_\bbC}:= \langle x, u\rangle + \langle y, v \rangle + i \langle y, u \rangle -i \langle x, v \rangle$. Then $\bH_\bbC$ is a complex Hilbert space. For $z=x+iy \in \bH_\bbC$ we call $x=\re(z)$ and $y=\im(z)$ the \emph{real} and \emph{imaginary} part of $z$. Furthermore, we denote by $\bar z:= \re(z)-i\im(z)$ the complex conjugate of $z$ and the imaginary subspace of $\bH$ by $i\bH=\{z \in \bH_\bbC: \re(z)=0\}$. The space of complex numbers with non-positive real part is denoted by $\bbC_-  = \{ c \in \bbC:\re (c) \le 0 \}$.

\subsection{Affine processes}
We are interested in \emph{homogeneous} infinite-dimensional continuous affine processes and introduce the following definition. While we do not aim at the greatest level of generality, we use a standard definition of affine processes. For a slightly more general approach (in finite dimensions) see \cite{KellerResselSchachermayerTeichmann2011}. The time-inhomogeneous case can be treated as in \cite{Filipovic05} and \cite{KRSchmidtWardenga}. 

Consider a closed subset $\cX\subset H$ which will serve as state space of our affine process and assume that the 
closure of the affine hull of $\cX$ is the full space $H$. Let $(\Omega,\cF,\bbF)$ be a filtered space on which a family of probability measures $(P_x)_{x \in \cX}$ is given. The filtration $\bbF$ is right-continuous and $P_x$-complete for all $x \in \cX$. Finally, consider a continuous process $X$ with values in $\cX$ and denote its transition kernel by
\begin{align*}
   p_t(x,A) = P_x(X_t \in A), 
\end{align*} 
for $t \ge 0, x \in \cX, A \in \ccB(H)$. We assume that the transition kernel is a Markov transition kernel, i.e.\ it satisfies the following properties (cf. \cite{ethier-kurtz-86})
\begin{enumeratei}
\item $x \mapsto p_t(x,A)$ is $\ccB(H)$-measurable for each $(t,A) \in \R_{\ge 0}\times \ccB(H)$,
\item $p_0(x,\{x\})=1$ for all $x \in \cX$,
\item $p_t(x,\cX)=1$ for all $(t,x) \in \R_{\ge 0}\times H$,
\item $p$ satisfies the \emph{Chapman-Kolmogorov equation}, i.e. for each $t,s \ge 0$ and $(x,A)\in H\times \ccB(H)$, it holds that
\begin{align}
p_{t+s}(x,A) = \int p_t(y,A) p_s(x,dy).
\end{align}
\end{enumeratei}

The affine property of the Markov process $X$ is characterized via its Fourier transform. 
The convex cone where the Fourier transform is defined by 
$$ \cU:= \big\{ u \in H_\bbC: \sup_{x \in \cX} \re(\langle u,x\rangle_{H_\bbC}) < \infty \big\}. $$
Then the function $\cX \ni x \mapsto e^{\langle u,x\rangle}$ is bounded if and only if $x \in \cU$. Moreover,  $iH \subset \cU$.

For a function $\phi:\R_{\ge 0} \to H$ the concepts of Fr\'echet and Gateaux differentiability coincide and we call $\phi$ \emph{differentiable} with derivative $D_t \phi, t \ge 0$ being a vector $D_t \phi(t) \in H$, if for every $t \ge 0$ it holds that 
$$ \lim_{\varepsilon \to 0} \frac{\parallel \phi(t+\varepsilon) - \phi(t) - \varepsilon D_t \phi(t)\parallel } {|\varepsilon|} = 0. $$

\begin{definition}\label{def:affine} An $H$-valued continuous process $X$ with transition kernel $p_t(x,A)$ is called \emph{affine} with state space $\cX$, if there exist functions $\phi : \R_{\ge 0} \times \cU \to \bbC$ and $\psi : \R_{\ge 0} \times \cU \to H_\bbC$ such that
\begin{enumeratei}
\item $\phi(\cdot, u)$ and $\psi(\cdot,u)$ are differentiable for each $u\in \cU$,
\item the derivatives $D_t \phi(t,u)$ and $D_t \psi(t,u)$ are jointly continuous, and
\item the Fourier-transform has  exponential affine dependence on the initial value, i.e. for all $t \ge 0, x \in \cX$, and $u \in \cU$ it holds that  
\begin{align}\label{eq:affine}
\int e^{\langle u,y \rangle_{H_\bbC}} p_t(x,dy) = \exp\big( \phi(t,u) + \langle \psi(t,u), x \rangle_{H_\bbC} \big). \end{align}
\end{enumeratei}
\end{definition}
Uniqueness of $\phi$ and $\psi$ holds under the normalization $\phi(0,u)=0$ and $\psi(0,u)=u$. Finite-dimensional affine processes can be viewed as a special case when $H=\R^n$. In this case, Definition \ref{def:affine} coincides with the affine class studied in \cite{KellerResselSchachermayerTeichmann2011}.

As a next step we study infinite-dimensional diffusions and classify the affine ones. First, we split the state space in the non-negative part and the unrestricted part. Note that in contrast to the usual procedure in finite dimensions, we gain additional freedom as the basis can be chosen in a suitable way. For any index set $K \subset \N$ we denote the canonical projection to the subspace $H_K$ by $\pi_K: x \mapsto \sum_{k \in K} \langle x, e_k \rangle e_k$ and for $x \in H$ we simply write $x_K=\pi_K x$.  Assume that the state space of $\cX$ is the direct sum 
\begin{align}\label{Dsep}
	\cX=H_I^+\oplus H_J
\end{align} 
where  $I,J \subset \N$ are two disjoint sets such that $I \cup J = \N$, and $H_I^+ := \{ \sum_{i \in I} \langle x, e_i \rangle e_i: x \in H, \langle x, e_k \rangle \ge 0 \}\subset H_I$. Then $\cX$ is a total set, i.e.\ the closure of its span is the full space $H$ and for any $x\in \cX$ we obtain  the unique decomposition $x=x_I+x_J$. 
Using this structural assumption on the state space $\cX$, the set $\cU$ can be determined precisely as follows: for $x \in H$ we write $x \le 0$ if $\la x, e_k \ra \le 0$ for all $k\in \NN$ and similar for $<$,$>$ or $\ge$. It turns out that under \eqref{Dsep},
\begin{align}
\label{Usep} 
	\cU=\big\{ u \in H_\bbC: \re(u_I) \le 0 \text{ and }\re(u_J)=0\big\}.
\end{align}
Moreover,  the finite-dimensional affine processes studied in \cite{DuffieFilipovicSchachermayer} can be viewed as special case with $H=\R^n$ and $\cX=\R_{\ge 0}^i \oplus \R^j$ and $i+j=n$.

\begin{remark}\label{rem2.1}
Fix $t \ge 0$. 
If $X$ is affine and the state space satisfies \eqref{Dsep}, then it follows from Equation \eqref{eq:affine} that, for all $x \in \cX$ and $u \in \cU$, 
	\begin{align*}
		e^{\re (\phi(t,u) + \langle \psi(t,u), x\rangle_{H_\bbC})}
		& = |e^{\phi(t,u) + \langle \psi(t,u), x\rangle_{H_\bbC}}| \\
		& \le \int |e^{\langle u, y\rangle_{H_\bbC}}| \, p_t(x,d y)  \le \int e^{\re\langle u, y\rangle_{H_\bbC}} p_t(x,d y)  \le 1,
	\end{align*}
	since $\re\langle u, y\rangle_{H_\bbC} \le 0$ for $u \in \cU$ and $y \in \cX$. Hence,  $\re (\phi(t,u) + \langle \psi(t,u), x\rangle_{H_\bbC}) \le 0$ for  all $x \in \cX$ and $u \in \cU$ which is equivalent to $(\phi(t,u),\psi(t,u)) \in \bbC_-\times \cU$ for all $u\in\cU$.
\end{remark}

We are interested in those Markov processes which are strong solutions of stochastic differential equations with respect to an infinite-dimensional Brownian motion. We follow the construction of a stochastic integral laid out in \cite{DaPratoZabczyk}. To this end,  denote the trace of a symmetric and non-negative operator $Q$ by $\Tr Q=\sum_{i=1}^\infty \langle Q e_i,e_i \rangle$ and call the operator $Q$ trace-class if $\Tr Q< \infty$.  Let $W$ be an $H$-valued $\bbF$-Brownian motion with covariance operator $\Sigma_W$, i.e.\ $\Sigma_W$ is a symmetric and non-negative definite operator $\Sigma_W$ with $\Tr \Sigma_W< \infty$. Then  there exists $\Sigma_W^{\nicefrac{1}{2}}$ such that $\Sigma_W = \Sigma_W^{\nicefrac{1}{2}} (\Sigma_W^{\nicefrac{1}{2}})^*$. 
Denote $H_0:= \Sigma_W^{\nicefrac{1}{2}} H$ and by $HS(H_0;H)$ the space of all Hilbert-Schmidt operators from $H_0$ to $H$, i.e.\ linear operators $Q$ such that $\sum_{i=1}^\infty \langle Q \Sigma_W^{\nicefrac{1}{2}} e_k, \Sigma_W^{\nicefrac{1}{2}}e_k \rangle^2 < \infty$. 

We assume, that for each $x_0\in \cX$, $X=X^{x_0}$ is the unique strong solution to the stochastic differential equation 
\begin{align}\begin{aligned}
dX_t & = \mu(X_t) dt + \sigma(X_t) dW_t, \\
X_0 &= x_0 \end{aligned}\label{SDE:X}
\end{align}
where $\mu: H \to H$ and $\sigma: H \to HS(H_0;H)$ are continuous; compare Theorem \ref{thm-main-ex} for precise conditions ensuring the existence of a unique strong solution. 

By $S(\cdot) := \sigma(\cdot) \Sigma_W \sigma(\cdot)^*$ we denote the dispersion operator of $X$, such that $d[X,X]_t = S(X_t) dt$.  The next result shows that $S(x)$ is a trace-class operator for each $x \in \cX$ and that $x \mapsto \Tr S(x)$ is a real-valued and continuous function.

\begin{lemma}\label{lem:alphax}
For each $x \in \cX$ the operator $S(x)$ is non-negative definite and trace-class. Moreover, the mapping $\Tr S(\cdot):H \to \R$ is continuous.
\end{lemma}
\begin{proof}
Note that $\Sigma_W$ is a symmetric, non-negative definite and trace-class operator. Then, it follows that for $x \in \cX$ and $h \in H$
$$ h S(x) h^* = (h \sigma(x)) \Sigma_W (h \sigma(x))^* \ge 0 $$
such that $S(x)$ is also symmetric and non-negative definite. We denote $Q=\Sigma_W^{\nicefrac{1}{2}}$ such that $\Sigma_W = Q Q^*$. From the cyclic property of the trace and the Cauchy-Schwarz inequality it follows that
\begin{align*}
   \Tr S(x) &= \Tr( (\sigma(x) Q) (\sigma(x) Q)^*  ) 
   = \Tr ( (\sigma(x) Q)^* (\sigma(x) Q)  ) \\  %
   & \le \Tr( \sigma(x) \sigma(x)^* )\cdot \Tr (Q Q^* ) = \Tr(\sigma(x) \sigma(x)^*) \, \Tr \Sigma_W < \infty,
\end{align*}
because for each $x\in H$, $\sigma(x) \in HS(H_0;H)$, and hence $\Tr(\sigma(x) \sigma(x)^*)<\infty$. The continuity from $S(x)$ now follows from the continuity of $\sigma(x)$.
\end{proof}

\begin{theorem}\label{thm:affine1}
Assume that the process $X$, given as unique strong solution of \eqref{SDE:X}, is affine. Then for all $x \in \cX$ it holds that
\begin{align} \label{affcoeffe} \begin{aligned}
\mu(x) &= m_0 + M x \\
S(x) &= n_0 + N x \end{aligned}
\end{align}
with $m_0 \in \cX$, $M \in L(H)$, $n_0 \in L(H)$ and $N \in L(H,L(H))$. Denote $n_k = N e_k$ and $m_k = M e_k$, $k=1,2,\dots$. 
The coefficients $n_k$ are symmetric, non-negative definite and trace-class operators which satisfy $n_j=0$ for all $j \in J$ and 
\begin{align} \label{condN}
   \sum_{i \in I} (\Tr n_i)^2 < \infty.
\end{align} 
The functions $\phi$ and $\psi_k(t,u):=\langle \psi(t,u),e_k\rangle_{H_\bbC}$, $k=1,2,\dots$ satisfy the general Riccati system
\begin{align}\label{Riccati1}\begin{split}
\partial_t \phi(t,u) &= \langle m_0, \overline{\psi(t,u)} \rangle_{{H_\bbC}} + \half \langle n_0 \psi(t,u), \overline{\psi(t,u)} \rangle_{H_\bbC} \\
\phi(0,u) &= 0 \end{split}\\
\label{Riccati2}\begin{split}
\partial_t \psi_k(t,u) &= \langle m_k, \overline{\psi(t,u)} \rangle_{H_\bbC} + \half \langle n_k \psi(t,u), \overline{\psi(t,u)} \rangle_{H_\bbC}, \quad k=1,2,\dots \\
\psi(0,u) &= u, \end{split}
\end{align}
for all $t \ge 0$ and $u \in \cU$. 
\end{theorem}

\begin{proof}
When $X$ is an affine process, then the processes
\begin{align*}
M_t^u := \exp\big( \phi(T-t,u)+ \langle \psi(T-t,u),X_t \rangle_{H_\bbC}\big), \quad 0 \le t \le T
\end{align*}
are martingales for all $u \in \cU$ since $M_t^u = \E[\exp\big(  \langle u,X_T \rangle_{H_\bbC} \big)|\cF_t] = \E[M_T^u|\cF_t]$.
Next, we apply the It\^o-formula, see Theorem 4.32 in \cite{DaPratoZabczyk}, to $M_t^u$ with $f(t,x)= \exp\big( \phi(T-t,u)+ \langle \psi(T-t,u),x \rangle_{H_\bbC}\big)$. Note that, by Lemma \ref{lem:alphax},
$$ \partial_t f(t,x) = f(t,x) \big( -\partial_t \phi(T-t,u) - \langle D_t \psi(T-t,u),x \rangle_{H_\bbC}). $$
Hence,
\begin{align} \label{temp313} 
dM_t^u = M_t^u( I_t dt + \langle \psi(T-t,u), \sigma(X_t) dW_t \rangle_{H_\bbC} ),
\end{align}
where the drift computes to
\begin{align}\label{eq:It}
I_t &= -\partial_t \phi(T-t,u) - \langle D_t \psi(T-t,u),X_t \rangle_{H_\bbC} + \langle \psi(T-t,u), \mu(X_t) \rangle_{H_\bbC} \\
&+ \half \sum_{k=1}^\infty \langle S(X_t) \psi(T-t,u), e_k \rangle_{H_\bbC} \langle \psi(T-t,u), e_k \rangle_{H_\bbC}. \nonumber
\end{align}
The infinite sum equals $ 
 \langle S(X_t) \psi(T-t,u), \overline{\psi(T-t,u)}  \rangle_{H_\bbC}$.
Moreover, the process $M^u$ is a martingale only if $I_t=0$ $dt\otimes dP$-almost surely. By continuity of $I$ it follows even that $I=0$ $P$-almost surely. Letting $t\to 0$, continuity of $X, \mu, S, \phi$, and $\psi$ implies that
\begin{align} \label{temp324}
	\partial_t \phi(t,u) + \langle D_t \psi(t,u),x \rangle_{H_\bbC} &=    \langle \psi(t,u), \mu(x) \rangle_{H_\bbC} 
+ \half \langle S(x) \psi(t,u), \overline{\psi(t,u)} \rangle_{H_\bbC}
\end{align}
holds for all $x \in \cX$ and all $t \ge 0$. The left hand side is an affine function of $x$, and hence the right-hand side is affine in $x$. Using that $\psi(0,u)=u$ we obtain that $\mu$ as well as $S$ are  affine functions of $x$,  such that representation \eqref{affcoeffe} follows. 
Continuity of $\mu$ yields that $M\in L(H)$. Moreover, by Lemma \ref{lem:alphax}, $S(x)$ is a symmetric, non-negative definite and trace-class operator for all $x \in \cX$. This gives that $n_j=N e_j = 0$ for all $j \in J$. Regarding \eqref{condN}, it follows
\begin{align*}
\Tr n_0 + \sum_{i \in I} \langle x,e_i \rangle_H \Tr n_i = \Tr S(x) < \infty 
\end{align*}
for all $x\in \cX$, because $S$ is trace-class. Define $T_n x := \sum_{i \in I, i \le n} \langle x,e_i \rangle_H  \Tr n_i$. Then $T_n \in L(H,\R)$ and 
$$ \sup_{n \in \N} \parallel T_n x \parallel \le \sum_{i \in I} | \langle x,e_i \rangle_H |\Tr n_i < \infty\quad \text{for all } x \in H. $$
By the uniform boundedness principle it follows that $\sum_{i \in I} (\Tr n_i)^2 = \sup_{n \in \N} \parallel T_n \parallel ^2< \infty$ such that \eqref{condN} follows. 

Finally, inserting \eqref{affcoeffe} into \eqref{temp324} and separating terms gives \eqref{Riccati1}-\eqref{Riccati2} since the affine hull of $\cX$ is the full space $H$, where again Lemma \ref{lem:alphax} was used. 
\end{proof}

The converse is solved in two steps. First, we derive some admissibility conditions for the coefficients of the Riccati equations \eqref{Riccati1}-\eqref{Riccati2}, which are equivalent to affinity in the canonical state space. %
Second, we show that these admissibility conditions are sufficient for existence and uniqueness of solutions for the Riccati equations.

\begin{proposition}\label{prop-2-4}
Assume that $X$ is a strong solution of \eqref{SDE:X}, $\mu$ and $S$ are affine in the sense of  \eqref{affcoeffe}
and the Riccati system \eqref{Riccati1}-\eqref{Riccati2} has a solution $(\phi,\psi)$ such that $\phi(t,u) + \langle \psi(t,u),x\rangle_{H_\bbC}$ has a non-negative real part for all $t\ge 0, u\in \cU$ and $x \in \cX$. Then $X$ is an affine process.  
\end{proposition}

\begin{proof}
	If \eqref{affcoeffe} and \eqref{Riccati1}-\eqref{Riccati2} hold such that $\phi(t,u) + \langle \psi(t,u),x\rangle_{H_\bbC}$ has a non-negative real part for all $t\ge 0, u\in \cU$ and $x \in \cX$, then it follows as in the proof of Theorem \ref{thm:affine1}, that  the drift $I$ of the process $M^u$ given in  \eqref{temp313} vanishes, such that
	$$
	  d M_t^u = M_t^u \, \langle \psi(T-t,u),\sigma(X_t) d W_t\rangle_{H_\bbC}.
	$$
	Hence,  $M_t^u$ is a continuous local martingale. 
    From the assumption that  $\re (\phi(T-t,u) + \langle \psi(T-t,u),X_t\rangle_{H_\bbC}) \le 0$, it follows that
    $$ M_t^u = \exp( \phi(T-t,u) + \langle \psi(T-t,u),X_t\rangle_{H_\bbC}) $$  
    is uniformly bounded by $1$ and hence $M^u$ is even a true martingale. Consequently, for all $t \ge 0$
     $$
	 E[ e^{\langle u,X_T \rangle_{H_\bbC}} |\cF_t ] = E[M_T^u |\cF_t] = M_t^u = \exp( \phi(T-t,u) + \langle \psi(T-t,u), X_t\rangle_{H_\bbC}).
	 $$
	 Then  \eqref{eq:affine} holds and $X$ is an affine process.
\end{proof}

The next result gives a partial answer to the solvability of the system of Riccati equations \eqref{Riccati1}-\eqref{Riccati2}. We start with some  notation. 
First, define 
	\begin{align*}
		H_\bbC^- & = \{ z=x+i y \in H_\bbC \ |\ \langle x,e_k\rangle_H \le 0 \text{ for all}\ k\in\N \}. 
	\end{align*}
As previously, $(H_\bbC^-)_I$ denotes the projection to the coordinates from set $I$, i.e. $ (H_\bbC^-)_I=  \{\sum_{i\in I} \langle z,e_i \rangle_{H_\bbC} e_i\ |\ z \in H_\bbC^-\}$.
Second, for the two index sets $K,L \subset \N$ and generic $A \in L(H)$ we denote $A_{KL}=\pi_K A|_{H_L}$. Then $A_{K \cup L}$ may be uniquely represented by the  $2\times2$ block operator matrix
			$$
			\begin{pmatrix}
				A_{KK} & A_{KL} \\
				A_{LK} & A_{LL} \\
			\end{pmatrix}.
			$$
			If the index sets are  singletons, we write $A_{kl}$ for $A_{\{k\}\{l\}}$.
Finally, for $x \in H$ we understand $x \le 0$ as $\langle x, e_k \rangle \le 0$ for all $k \in \N$ and $x\ge0$, $x<0$, $x>0$, $x\nless0$, $x\ngtr0$, $x=0$ in the same manner.

\begin{proposition}\label{prop2.5}
Assume that \eqref{Dsep} and  the following admissibility conditions hold:
\begin{align}
	m_0 \in \cX,\ &m_i \in H_{I \backslash\{i\}}^+ \oplus H_{J \cup \{i\} } %
    \text{ for } i \in I, \text{ and }m_j \in H_J \text{ for } j \in J \label{prop1:1}\\
    & \hspace{-5mm} || \sum_{k \ge 1} m_k  \langle \cdot , e_k \rangle || < \infty  \label{prop1:1a}\\
	n_k\in L(H) & \text{ is symmetric, non-negative definite and of trace class, } k \in \N, \label{prop1:2}\\
	n_j &= 0 \text{ for } j \in J, \label{prop1:3}\\
    n_{0,II} &=0, \notag\\
    n_{0,IJ} &= n_{0,JI}^* = 0, \notag\\
    n_{0,JJ} & \text{ is symmetric, non-negative definite and of trace class}, \notag\\
    n_{i,\{kl\}} &= \begin{cases} \ge 0 & \text{ if }i=k=l, \\ =0 & \text{ otherwise}, \end{cases} \quad \text{ for }i,k,l \in I,  \label{prop1:7}\\
    n_{i,IJ} &= n_{i,JI}^* \notag\\
    n_{i,JJ} & \text{ is symmetric, non-negative definite and of trace class}, \notag\\
    & \hspace{-5mm}\sum_{i\in I} \|n_i\|^2  < \infty. \label{prop1:10}
\end{align}

Then the general Riccati system \eqref{Riccati1}-\eqref{Riccati2} has a unique solution $(\phi(\cdot,u),\psi(\cdot,u)) : \R_+ \to \bbC_- \times (H_\bbC^-)_I \oplus iH _J$ for each
$u \in (H_\bbC^-)_I \oplus i H_J$. 
\end{proposition}

These conditions directly correspond to the well-known conditions in the finite-dimensional case, see \cite{Filipovic2009}, with additional assumptions on summability of certain coefficients, \eqref{prop1:1a}, and \eqref{prop1:10}. Regarding \eqref{prop1:1}, this can be seen as follows: note that for  $i \in I$ and $j \in J$, $M_{II}e_i = \pi_I M |_{H_I} e_i = \pi_I m_i \in H_{I\setminus\{i\}}^+ \oplus H_{\{i\}}$ as well as $M_{IJ}e_j = \pi_I M |_{H_J} e_j = \pi_I m_j = 0$ because $m_j\in H_J$. This corresponds to the  condition of $\mathcal{B}_{II}$ having nonnegative off-diagonal elements and $\cB_{IJ}=0$ of Theorem 10.2 in \cite{Filipovic2009} (in the notation used there).

Condition \eqref{prop1:10} is always satisfied in the finite-dimensional case and appears here for the first time in literature.
Denote the eigenvalues of the trace-class operator $n$ by $\lambda_i, i \ge 1$. If $n$ is also symmetric and non-negative definite, then %
 \begin{align}
     \| n \| \le \sum_{i \ge 1} | \lambda_i | = \Tr n.
 \end{align}
Hence, a sufficient criterion for $\sum_{i\in I} \|n_k\|^2<\infty$ is $\sum_{i\in I} (\Tr n_i)^2<\infty$.

The proof is separated in a number of smaller results. Set $f(\xi) = \frac12 \sum_{i \in I}^\infty\la n_i \xi,\overline{\xi} \ra_{H_\bbC}\ \! e_i$. Then the Riccati equations in \eqref{Riccati2} are equivalent to  the following semilinear evolution equation  
\begin{equation}
		\begin{split}
			\partial_t \psi(t,u) & = M^\top \psi(t,u) + f(\psi(t,u)) , \quad u \in \cU,\ t \ge 0
		\end{split}
		\label{equ:ats:psi}
\end{equation}
with initial condition $\psi(0,u)  = u$. Such equations have been studied in \cite{Weissler1979} and Theorem 1 therein yields  the following result.

\begin{lemma}\label{lem:2.6}
	Assume that \eqref{prop1:1a} and \eqref{prop1:10} are satisfied. Then, for each $u \in \cU$, \eqref{equ:ats:psi} has a unique solution $\psi(t,u)$  on some interval $[0,T_u)$ with existence time $T_u \in (0,\infty].$
\end{lemma}

\begin{proof}
We will first show that $K_t := e^{tM^\top} f,\ t \ge 0$ is locally Lipschitz-continuous. To this end, note that for $\xi$ and $\eta$ in the domain of $f$,
\begin{align*}
	\|e^{tM^\top}f(\xi) - e^{tM^\top}f(\eta)\|_{H_\bbC}^2 &\le \frac14 e^{2\|M\| t} \sum_{i \in I} \left|\la n_i\xi,\overline{\xi} \ra_{H_\bbC} - \la n_i\eta,\overline{\eta} \ra_{H_\bbC}\right|^2 \\
		&= \frac14 e^{2\|M\| t} \sum_{i \in I} \left|\la n_i\xi,\overline{\xi} \ra_{H_\bbC} - \la n_i\eta,\overline{\xi} \ra_{H_\bbC} + \la n_i\eta,\overline{\xi} \ra_{H_\bbC} - \la n_i\eta,\overline{\eta} \ra_{H_\bbC}\right|^2 \\
		&= \frac14 e^{2\|M\| t} \sum_{i \in I} |\la n_i(\xi+\eta),\overline{(\xi-\eta)} \ra_{H_\bbC}|^2 \\
		&\le \frac14 e^{2\|M\| t} \sum_{i \in I} \| n_i\|^2\|\xi+\eta\|_{H_\bbC}^2\|\xi-\eta\|_{H_\bbC}^2. 
\end{align*}

Hence,  for each $t \ge 0$, $K_t$ is locally Lipschitz-continuous by \eqref{prop1:10} and its  Lipschitz constant on  $U_\alpha=\{x \in H:\parallel x \parallel \le \alpha\}$, $\alpha >0$ is bounded by 
$ \alpha e^{\parallel M \parallel t}  ( \sum_{i \in I} \| n_i \|^2)^{\nicefrac 12}. $ 

Theorem 1 in \cite{Weissler1979} now yields that, for each $u \in \cU$, equation \eqref{equ:ats:psi} possesses a unique solution $\psi(t,u)$  on some interval $[0,T_u)$ with $0<T_u\le \infty$ and the proof is finished. \end{proof}
\begin{lemma}
	Assume that the admissibility conditions \eqref{prop1:1a} - \eqref{prop1:10} are satisfied. Then, for all $t\in[0,T_u)$ and $u\in\cU$ it holds that the unique solution of \eqref{equ:ats:psi}, $\psi(t,u)$, satisfies that $\psi(t,u) \in \cU$.
\end{lemma}

\begin{proof} 
To begin with, we note that by Lemma \ref{lem:2.6}, \eqref{equ:ats:psi}, has, for $u \in \cU$ a unique solution $\psi(t,u)$ on $[0,T_u)$.
To show that $\psi(t,u) \in \cU$, we utilize \eqref{Usep}, hence, have to show that $\re(\psi_J(t,u))=0$ and $\re(\psi_I(t,u)) \le 0$.

First, for $j \in J$, we obtain by \eqref{prop1:3}, that the projection $\psi_J(t,u) = \pi_J \psi(t,u)$ satisfies the autonomous equation
\begin{align*}
	\psi_J(t,u) &= M_{JJ}^* \psi_J(t,u), \quad t\ge0, \ u\in\cU,
\end{align*}
with $\psi_J(0,u) =  u_J$.  The unique solution of this equation is given by 
$$
\psi_J(t,u) = e^{tM_{JJ}^*} \, u_J,\quad t\ge0.
$$
From \eqref{Usep} it follows $\re (u_J)=0$ and hence $\re(\psi_J(t,u))=0$. 

As a second step we show that $\re(\psi_I(t,u)) \le 0$ which requires more work.
We start with the observation that, for $i \in I$, 
\begin{align}
 \partial_t\re(\psi_i(t,u)) 
		&= \langle m_i,\re (\psi(t,u))\rangle \nonumber\\
        &+ \frac12\langle n_i\re(\psi(t,u)), \re(\psi(t,u))\rangle 
		- \frac12\langle n_i\im(\psi(t,u)), \im(\psi(t,u))\rangle \nonumber\\
        & \le \langle m_i,\re (\psi(t,u))\rangle + \frac12 n_{i,\{ii\}}(\re(\psi_i(t,u)))^2,
        \label{temp477}
\end{align}
using \eqref{prop1:3} and \eqref{prop1:7}. 

Next, consider $\varepsilon>0$ and $u \in \cU$  such that %
$\re (u_i) < -\varepsilon$ for all $i \in I$. Let 
\begin{align}\label{def:Tustrich}
   T'_u =\inf\{t\in[0,T_u) : \exists i \in I \text{ s.t. }\re (\psi_i(t,u)) \ge  0\}
\end{align}
with the convention that $\inf \emptyset=\infty$.
Since $\re(\psi_i(t,u))$ is continuous at $t=0$ and $\re(\psi_i(0,u))$ bounded away from zero, $T'_u>0$. 
Hence,  for $t\in[0,T'_u)$, it follows that $\re(\psi_{I'}(t,u))<0$ for any subset $I'\subset I$. 

By Assumption \eqref{prop1:1}, $m_i \in H_{I \backslash\{i\}}^+ \oplus H_{J \cup \{i\} }$. Hence, there exist $m_i'\in H_{I \backslash\{i\}}^+$ and $m_i'' \in H_{J}$ such that $m_i=m_{i}'+m_{i,i}+m_{i}''$. Then $\langle m_i',\re (\psi_{I\setminus\{i\}}(t,u))\rangle \le 0$, $\langle m_i'',\re (\psi_J(t,u))\rangle=0$ and therefore
\begin{align*}
	\langle m_i,\re (\psi(t,u))\rangle &\le \langle m_{i,i},\re (\psi_i(t,u))\rangle.
\end{align*}
Together with \eqref{temp477} and $C_i = \frac12 \max\{n_{i,\{ii\}}, |m_{i,i}|\}$, we are able to achieve the following estimate, 
\begin{align*}
	\partial_t \re(\psi_i(t,u))
	&\le \langle m_{i,i},\re (\psi_i(t,u))\rangle + \frac12 n_{i,\{ii\}}(\re(\psi_i(t,u)))^2  \\
	&\le C_i\big((\re (\psi_i(t,u)))^2 - 2\re (\psi_i(t,u))\big),
\end{align*}
where we used $\re (\psi_i(t,u))<0$. 
By the comparison theorem,  \cite{BirkhoffRota1989}(Chapter 1, Theorem 7), 
$\re(\psi_i(t,u))\le g(t,u_i,C_i)$ for all $t\in[0,T'_u)$, where $g(t,u,C)=:g(t)$ solves
\begin{align*}
	\partial_tg(t) &= C(g(t)^2 - 2g(t)), \\
	g(0) &= \re(u),
\end{align*}
with $C \ge 0$ and $\re(u)< 0$.
The unique solution of this Riccati equation is given by $g(t) = 2u (2e^{2C t} - u(e^{2Ct}-1))^{-1}$. The function $g(\cdot,u,C)$ stays negative on the whole real line when $\re(u)<0$. Moreover, $g$ is increasing in $u$ and $C$ such that we obtain that  
$$ 
   \re(\psi_I(t,u)) \le \sup_{i \in I} g(t,u_i,C_i) 
   \le  g(t,-\varepsilon,C^*)<0
$$ 
for  $t\in[0,T'_u)$ where 
$C^* = \sup_{i \in I} C_i \le \frac12 \sup_{i\in I}(\|n_i\| + \|m_i\|) \le \frac12 ((\sum_{i \in I} \|n_i\|^2)^{\nicefrac12} + \|M\|) < \infty$ by \eqref{prop1:1a} and \eqref{prop1:10}. Using continuity of $\re(\psi_I(t,u))$, we obtain that at $t=T_u'$, $\re(\psi_I(t,u))<0$, if $T_u'<\infty$.  By the very definition of $T_u'$ in \eqref{def:Tustrich}, this implies that $T_u'=\infty$. 

Summarizing, we obtained up to now that  for $u \in \cU$ with $\re(u_I)<-\varepsilon$ it follows that $\re(\psi_I(t,u))\le 0$ for all $t \in [0,T_u)$.
The next step is to extend this result to all $u \in \cU$.

In this regard, consider arbitrary $u \in \cU$, a sequence $(\varepsilon_n)\downarrow 0$ and a sequence $u_n \to u$ satisfying 
   $\re (u_n) < -\varepsilon_n$ for all $n \ge 1$. 
By part (v) of Theorem 1 in  \cite{Weissler1979}, $\psi(t,u)$ is Lipschitz continuous on some neighborhood of $u$, uniformly on each compact interval $[0,T]$, $T<T_u$. Therefore, $\psi(t,u_n) \to \psi(t,u)$ for each $t \in [0,T_u)$. Hence, 
   $$\re (\psi_I(t,u)) = \lim_{n \to \infty} \re (\psi_I(t,u_n))\le 0, \qquad t\in[0,T_u) $$ 
and the claim is proved.
\end{proof}

\begin{lemma}\label{lemma 2.8}
Assume that the admissibility conditions \eqref{prop1:1a} - \eqref{prop1:10} are satisfied. Then, for all $t\in[0,T_u)$ and $u\in\cU$ it holds that the unique solution of \eqref{equ:ats:psi}, $\psi(t,u)$, satisfies the following inequality
	\begin{align}
	\label{gronwall estimate psi}
	\|\psi_I(t,u)\|_{H_\bbC}^2 &\le \|u_I\|_{H_\bbC}^2 + C(1+\|u_I\|_{H_\bbC}^2)\int_0^t h_u(s) e^{C\int_s^t h_u(r) d r} d s,
\end{align}
where $h_u(t)=\big(1+\|\psi_J(t,u)\|_{H_\bbC}^2+\|\psi_J(t,u)\|_{H_\bbC}^4\big)$ with $\psi_J(t,u) = e^{tM_{JJ}^*} \, u_J$ and $C=\sum_{i\in I} \|A_i\|^2 + \|M\|^2 + \frac72$.
\end{lemma}

The proof of this lemma is relegated to the appendix.
Finally, we show that the unique solution exists on the whole real line, thus completing the proof of Proposition \ref{prop2.5}.

\begin{proof}[Proof of Proposition \ref{prop2.5}]
First, we show that $T_u=\infty$. The proof bases on result (iv) of Theorem 1 in \cite{Weissler1979}, saying that  $\lim_{t\to T_u}\|\psi_I(t,u)\|_{H_\bbC} = \infty$ if $T_u<\infty$. 
In this regard, note that the right hand side of Equation \eqref{gronwall estimate psi} is finite for all $t \ge 0$. Hence,  the existence time $T_u$ of $\psi_I(t,u)$ for $u\in\cU$ must be infinite, i.e.\ $T_u=\infty$. 

This shows existence and uniqueness regarding $\psi$. Existence and uniqueness for $\phi$ directly follow by integration.
At last, we  show $\re\phi(t,u) \le 0$ for all $t\ge0$ and $u\in\cU$: integrate the real part of \eqref{Riccati1} and consider the admissibility conditions to get
\begin{align*}
	& \re\phi(t,u) \\
	&= \int_0^t \langle m_0,\re\psi(s,u) \rangle + \frac12 \langle n_0 \re\psi(s,u),\re\psi(s,u)\rangle - \frac12 \langle n_0 \im\psi(s,u),\im\psi(s,u)\rangle d s \\
	&= \int_0^t \langle m_{0,I},\re\psi_I(s,u) \rangle - \frac12 \langle n_{0,JJ} \im\psi_J(s,u),\im\psi_J(s,u)\rangle d s \le 0.
\end{align*}
for all $t\ge0$ and $u\in\cU$. 
\end{proof}

\begin{proposition}
Assume that $X$ is a strong solution of \eqref{SDE:X}, $\mu$ and $S$ are affine  as in \eqref{affcoeffe},
and that the Riccati system \eqref{Riccati1}-\eqref{Riccati2} has a solution $(\phi(t,u),\psi(t,u)) \in \bbC_- \times \cU$ for all $t\ge 0, u\in \cU$ and $x \in \cX$. Then the admissibility conditions in Proposition 3.5 hold.
\end{proposition}
\begin{proof}
	First of all, by Lemma \ref{lem:alphax}, $S(x)$ is a symmetric, non-negative definite and trace-class operator for all $x \in \cX$. This gives that $n_j=N e_j = 0$ for all $j \in J$. Moreover, from the Riccati equations \eqref{Riccati1}-\eqref{Riccati2} we obtain that	\begin{align}
		\partial_t \re\phi(0,u) &= \langle m_0, v \rangle + \frac12 \langle n_0 v, v  \rangle - \frac12 \langle n_0 w, w \rangle, \label{ReRic1} \\
		\partial_t \re\psi_i(0,u) &= \langle m_i, v \rangle + \frac12 \langle n_i v, v  \rangle - \frac12 \langle n_i w, w \rangle, \label{ReRic2I} \\
		\partial_t \re\psi_j(0,u) &= \langle m_j, v \rangle, \label{ReRic2J}
	\end{align}
	where ${u \in \cU,\ } $ and we set $ v=\re u, w=\im u$. 
	From \eqref{Usep}, together with \eqref{Riccati2}, we obtain from \eqref{ReRic2J} that $\re \psi_j(\cdot,u) \equiv 0$ for all $j\in J$. This implies  that {$\langle m_j, v\rangle = 0$. Again from \eqref{Usep}  we obtain that $v_J = 0$ while $v_I \le 0$. Hence,  $m_j\in H_J$} for all $j\in J$. 
	
Next, we consider \eqref{ReRic2I}. As already noted in  Remark \ref{rem2.1}, $\re \psi_i(\cdot,u) \le 0$ for all $i\in I$, such that $\partial_t \re \psi_i(0,u) \le 0$ whenever $\re \psi_i(0,u)= \re u_i = v_i =  0$.
	Choose $u=v+iw$ such that $v_{I\setminus\{i\}} < 0$, {$v_{J\cup\{i\}} = 0$} and $w=0$. Substituting such $u$'s into \eqref{ReRic2I} leads to
	\begin{equation}
		0 \ge \partial_t \re \psi_i(0,u) = \langle m_i, v_{I\setminus\{i\}}\rangle + \frac12 \langle n_i v_{I\setminus\{i\}}, v_{I\setminus\{i\}} \rangle.
		\label{temp333}
	\end{equation}
	This implies that $ \langle n_i v_{I\setminus\{i\}}, v_{I\setminus\{i\}} \rangle = 0$: indeed,  if $ \langle n_i v_{I\setminus\{i\}}, v_{I\setminus\{i\}} \rangle \neq 0$, there would exist a $v_{I\setminus\{i\}},$ such that  
	$$
	\langle n_i v_{I\setminus\{i\}}, v_{I\setminus\{i\}} \rangle > 0,
	$$
	and for $\gamma > 0$ large enough, \eqref{temp333} would lead to a contradiction that
	$$
	0 \ge \partial_t \re \psi_i(0,\gamma u) = \langle m_i, v_{I\setminus\{i\}}\rangle \gamma + \frac12 \langle n_i v_{I\setminus\{i\}}, v_{I\setminus\{i\}} \rangle \gamma^2 > 0.
	$$
	Now that we have shown $n_{i,I\setminus\{i\}I\setminus\{i\}} = 0$, it follows from the non-negative definiteness of $n_i$ that
	$$
	n_{i,\{kl\}} = \begin{cases} \ge 0 & \text{ if }i=k=l, \\ =0 & \text{ otherwise}, \end{cases} \quad \text{ for }i,k,l \in I.
	$$
	The rest conditions on $n_i$, such as $n_{i,IJ} = n_{i,JI}^*$ and
	$n_{i,JJ}$ is symmetric, non-negative definite and of trace class, can be easily seen from its non-negative definiteness as well. Furthermore, because of $n_{i,I\setminus\{i\}I\setminus\{i\}} = 0$, \eqref{temp333} gives
	$$
	\langle m_i, v_{I\setminus\{i\}} \rangle \le 0.
	$$
	Then we conclude that $m_i \in H_{I \backslash\{i\}}^+ \oplus H_{J \cup \{i\} }$, since $v_{I\setminus\{i\}}$ is chosen to be arbitrarily negative. Finally, we look at \eqref{ReRic1}. Since $\re \phi(\cdot,u) \le 0$ and $\re \phi(0,u)=0$, we may employ the same reason as for  \eqref{ReRic2I} to detect $\partial_t \re \phi_i(0,u) \le 0$ for all $u \in \cU$. Especially, we choose $u=v+iw$ with $v_I < 0$, $v_J = 0$ and $w=0$ and get
	\begin{equation}
		0 \ge \partial_t \re \phi_i(0,u) = \langle m_0, v_I\rangle + \frac12 \langle n_0 v_I, v_I\rangle.
		\label{temp444}
	\end{equation}
	An analogous argument applied to \eqref{temp333} shows that $n_{0,II} = 0$. Besides, the affine form condition tells that $n_0$ is a symmetric, non-negative definite, trace-class operator, which implies that $n_{0,JJ}$ must be such one as well and $n_{0,IJ} = n_{0,JI}^*=0$ due to $n_{0,II} = 0$. Moreover, such an $n_0$ turns \eqref{temp444} to
	$$
	\langle m_0, v_I\rangle = \partial_t \re \phi_i(0,u) \le 0.
	$$
	Then $m_0$ must be an element in $\cX$, because $v_I < 0$ is arbitrary.
\end{proof}

\begin{remark}
	Consider the canonical state space $\cX$ and assume $X$ to be a strong solution of \eqref{SDE:X}. Then the affinity property of $X$ is equivalent to the admissibility conditions. The sufficiency is deduced by Theorem \ref{thm:affine1} and Lemma \ref{lem:2.6} and
	the necessity results from Proposition \ref{prop2.5} and Proposition \ref{prop-2-4}. Both Theorem \ref{thm:affine1} and Proposition \ref{prop2.5} indicate that the both equivalent statements imply the existence and uniqueness of solutions of the Riccati equations \eqref{Riccati1}-\eqref{Riccati2}.
\end{remark}

\begin{remark}\label{rem-law}
By Theorem \ref{thm:affine1} the parameters $m_0,M,n_0,N$ in (\ref{affcoeffe}) determine the law of the process $X$. Indeed, these parameters determine the functions $\phi(\cdot,u) : \R_{\ge 0} \to \bbC$ and $\psi(\cdot,u) : \R_{\ge 0} \to H_{\bbC}$ as solutions of the Riccati equations (\ref{Riccati1}) and (\ref{Riccati2}) for all $u \in \cU$, and hence by (\ref{eq:affine}) for all $0 \leq s < t$ and $u,v \in \cU$ we have
\begin{align*}
&\bbe \Big[ e^{\la u,X_s \ra_{H_{\bbC}} + \la v,X_t \ra_{H_{\bbC}}} \Big] = \int_H \int_H e^{ \la u,y \ra_{H_{\bbC}} + \la v,z \ra_{H_{\bbC}} } p_{t-s}(y,dz) p_s(x,dy)
\\ &= \int_H \bigg( \int_H e^{ \la v,z \ra_{H_{\bbC}} } p_{t-s}(y,dz) \bigg) e^{ \la u,y \ra_{H_{\bbC}} } p_s(x,dy)
\\ &= \int_H \bigg( \exp \big( \phi(t-s,v) + \la \psi(t-s,v),y \ra_{H_{\bbC}} \big) \bigg) e^{ \la u,y \ra_{H_{\bbC}} } p_s(x,dy)
\\ &= \exp \big( \phi(t-s,v) \big) \int_H e^{ \la \psi(t-s,v) + u,y \ra_{H_{\bbC}} } p_s(x,dy)
\\ &= \exp \big( \phi(t-s,v) \big) \exp \big( \phi(s,u+\psi(t-s,v)) + \la \psi(s,u+\psi(t-s,v)),x \ra_{H_{\bbC}} \big),
\end{align*}
and analogously for every finite dimensional family $(X_{t_1},\ldots,X_{t_n})$. In particular, the law of $X$ stays invariant under transformations of the volatility $\sigma$ which provide the same dispersion operator $S$.
\end{remark}

\section{Existence of affine processes}\label{sec-existence}

The goal of this section is to provide an existence result for affine processes on Hilbert spaces in the spirit of strong solutions to infinite dimensional SDEs. In Subsection \ref{sec-ex-form} we will introduce the general framework and formulate the existence result; see Theorem \ref{thm-main-ex} below. Afterwards, Subsection \ref{sec-ex-proof} is devoted to its proof.

\subsection{Formulation of the existence result}\label{sec-ex-form}

Recall that $H$ is a separable Hilbert space with orthonormal basis $(e_k)_{k \in \bbn}$, and that the state space satisfies $\calx = H_I^+ \oplus H_J$. Starting point is the SDE (see \eqref{SDE:X})
\begin{align}\label{SDE-affine}
\left\{
\begin{array}{rcl}
dX_t & = & \mu(X_t)dt + \sigma(X_t) dW_t
\\ X_0 & = & x_0,
\end{array}
\right.
\end{align}
where $\mu : \calx \to H$ and $\sigma : \calx \to L_2(U_0,H)$ are continuous,  
 $W$ is an $U$-valued Wiener process on a separable Hilbert space $U$ with some covariance operator $\Sigma_W \in L_1^{++}(U)$, and the space $U_0 := \Sigma_W^{\nicefrac{1}{2}}(U)$ is the separable Hilbert space defined according to Lemma \ref{lemma-zusammenziehen}. We define the continuous mapping $S : \calx \to L_1^+(H)$ as
\begin{align}\label{S-sigma}
S(x) := \sigma(x) \Sigma_W^{\nicefrac{1}{2}} \big( \sigma(x) \Sigma_W^{\nicefrac{1}{2}} \big)^* \quad \text{for all $x \in \calx$.}
\end{align}
In the light of Theorem \ref{thm:affine1}, we assume that with $m_0 \in H$ and $M \in L(H)$,
\begin{align}\label{mu-affine}
\mu(x) = m_0 + M x \quad \text{for all $x \in \calx$,}
\end{align}
and with  $n_0 \in L_1^+(H)$ and $N \in L(H,L_1(H))$ 
\begin{align}\label{S-affine}
S(x) = n_0 + N x \quad \text{for all $x \in \calx$.}
\end{align}
Moreover, we assume that $n_0$ is self-adjoint, and that for every $x \in \calx$ the operator $Nx$ is self-adjoint with $Nx \in L_1^+(H)$. 

To ensure that the closed convex cone $\calx$ is invariant for the SDE (\ref{SDE-affine}), we  assume that  $\mu$ is  \emph{inward pointing} at boundary points of $\calx$, i.e.
\begin{align*}
\la \mu(x),\eta \ra_H \geq 0 \quad \text{for all $x \in \calx$ and all $\eta \in H_I^+$ with $\la x,\eta \ra_H = 0$,}
\end{align*}
and that the mapping $\sigma$ is  \emph{parallel to the boundary} at boundary points of $\calx$, i.e.
\begin{align}\label{sigma-par-def}
\la \sigma(x),\eta \ra_H = 0 \quad \text{for all $x \in \calx$ and all $\eta \in H_I^+$ with $\la x,\eta \ra_H = 0$,}
\end{align}
where we note that $\la \sigma(x), \eta \ra_H$ is an operator from $L_2(U_0,\bbr)$. 

For a linear operator $T \in L(H)$ we introduce the notations $T_I := \pi_I T$, $T_J := \pi_J T$ and $T_{II} := T_I|_{H_I}$, $T_{IJ} := T_J|_{H_I}$, $T_{JI} := T_I|_{H_J}$, $T_{JJ} := T_J|_{H_J}$. We define the sequences $\lambda = (\lambda_i)_{i \in I} \subset \bbr_+$ and $\kappa = (\kappa_i)_{i \in I} \subset \bbr_+$ as
\begin{align*}
\lambda_i := \| S(e_i)_{II} \, e_i \|_H \quad \text{and} \quad \kappa_i := \| S(e_i)_{IJ} \, e_i \|_H \quad \text{for each $i \in I$.}
\end{align*}
As we will show, we have $\lambda \in \ell^2(I)$. As a consequence, there exists a sequence $\nu = (\nu_i)_{i \in I} \subset (0,\infty)$ such that $\nu_i \to 0$ and $( \lambda_i / \nu_i )_{i \in I} \in \ell^2(I)$. Let $T \in K^{++}(H_I)$ be the compact linear operator with representation
\begin{align}\label{T-compact-repr}
Tx = \sum_{i \in I} \nu_i \la x,e_i \ra_H \, e_i \quad \text{for each $x \in H_I$,}
\end{align}
and let $H_{I,0} := T(H_I)$ be the retracted subspace with compact embedding defined according to Lemma \ref{lemma-zusammenziehen}. Furthermore, we set $H_{I,0}^+ := T(H_I^+)$ and $\calx_0 := H_{I,0}^+ \oplus H_J$. In addition, we require the following.

\begin{assumption}\label{ass-ex}
We suppose that 
\begin{enumerate}[(i)]
\item  $U = H$, and  $\Sigma_W$ has a diagonal structure along the orthonormal basis $( e_k )_{k \in \bbn}$,

\item for each $x \in \calx$ the operator $\sigma(x) \Sigma_W^{\nicefrac{1}{2}}$ is self-adjoint,

\item with $I_{>0} := \{ i \in I : \lambda_i > 0 \}$ it holds that
\begin{align}\label{frac-in-l2}
( \kappa_i / \lambda_i )_{i \in I_{>0}} \in \ell^2(I_{>0}),
\end{align}
\item and that
\begin{align}\label{ass-prin-drift}
&m_{0,I} \in H_{I,0}^+ \quad \text{and} \quad M_{II} T = T M_{II}, \\
\label{trace-cond-drift}
&\sum_{i \in I} \| M_{II}^i \|_{H_I'} < \infty,
\end{align}
where for each $i \in I$ the continuous linear functional $M_{II}^i \in H_I'$ is given by
\begin{align*}
M_{II}^i x := \la M_{II} x, e_i \ra_H, \quad x \in H_I.
\end{align*}
\end{enumerate}
\end{assumption}

These conditions do not mean severe restrictions. Indeed, the first condition means that the state space of the Wiener process is the same as the state space of the SDE (\ref{SDE-affine}), and that its covariance operator has a diagonal form with respect to the given orthonormal basis. This is also typically assumed in finite dimension. The second condition means that for each $x \in \calx$ we have $S(x)^{\nicefrac{1}{2}} = \sigma(x) \Sigma_W^{\nicefrac{1}{2}}$, and hence
\begin{align}\label{sigma-S}
\sigma(x) = S(x)^{\nicefrac{1}{2}} \Sigma_W^{-\nicefrac{1}{2}} \quad \text{for all $x \in \calx$.}
\end{align}
As mentioned in Remark \ref{rem-law}, other choices of the volatility $\sigma$ with the same dispersion operator $S$ do not change the law of the solution. Condition (\ref{frac-in-l2}) ensures that we can find a linear transformation $\Lambda \in L(H)$ with $\Lambda(\calx) = \calx$ such that for the transformed SDE 
\begin{align}\label{SDE-affine-Y}
\left\{
\begin{array}{rcl}
dY_t & = & \bar{\mu}(Y_t)dt + \bar{\sigma}(Y_t) dW_t
\\ Y_0 & = & y_0,
\end{array}
\right.
\end{align}
corresponding to $Y = \Lambda X$, the drift $\bar{\mu} : \calx \to H$ has a decomposition
\begin{align}\label{mu-decomp}
\bar{\mu}(y) = \bar{\mu}_{II}(y_I) + \bar{\mu}_J(y), \quad y \in \calx
\end{align}
with affine mappings $\bar{\mu}_{II} : H_I^+ \to H_I$ and $\bar{\mu}_J : \calx \to H_J$, and the volatility $\bar{\sigma} : \calx \to L_2(U_0,H)$ has a block diagonal structure
\begin{align}\label{sigma-decomp}
\bar{\sigma}(y)u = \bar{\sigma}_{II}(y_I)u_I + \bar{\sigma}_{JJ}(y_I)u_J, \quad y \in \calx \text{ and } u \in U_0
\end{align}
with mappings $\bar{\sigma}_{II} : H_I^+ \to L_2(U_{I,0},H_I)$ and $\bar{\sigma}_{JJ} : H_I^+ \to L_2(U_{J,0},H_J)$. This allows us to express the transformed SDE (\ref{SDE-affine-Y}) by the two coupled SDEs
\begin{align}\label{SDE-affine-Y-I}
\left\{
\begin{array}{rcl}
dY_{I,t} & = & \bar{\mu}_{II}(Y_{I,t}) dt + \bar{\sigma}_{II}(Y_{I,t}) dW_t
\\ Y_{I,0} & = & y_{0,I}
\end{array}
\right.
\end{align}
and
\begin{align}\label{SDE-affine-Y-J}
\left\{
\begin{array}{rcl}
dY_{J,t} & = & \bar{\mu}_J(Y_t) dt + \bar{\sigma}_{JJ}(Y_{I,t}) dW_t
\\ Y_{J,0} & = & y_{0,J},
\end{array}
\right.
\end{align}
and then our task is essentially reduced to solving the SDE (\ref{SDE-affine-Y-I}), which is feasible by virtue of condition (\ref{ass-prin-drift}). The condition (\ref{trace-cond-drift}) ensures pathwise uniqueness. Now, our main result of this section reads as follows. Concerning the notion of a unique strong solution starting in $\calx_0$, we refer to Appendix \ref{app-SDE}.

\begin{theorem}\label{thm-main-ex}
Suppose that Assumption \ref{ass-ex} is fulfilled. Then the  SDE \eqref{SDE-affine} has a unique strong solution starting in $\calx_0$.
\end{theorem}

If condition \eqref{trace-cond-drift} does not hold,  we still obtain  existence of a weak solution, but pathwise uniqueness might not be satisfied.

\subsection{Proof of the existence result}\label{sec-ex-proof}

The goal of this subsection is to provide the proof of Theorem \ref{thm-main-ex}. The main idea is to apply our version of the Yamada-Watanabe theorem (see Theorem \ref{thm-YW}). As already mentioned, after a suitable transformation we may consider the two coupled SDEs (\ref{SDE-affine-Y-I}) and (\ref{SDE-affine-Y-J}). This transformation procedure is similar to that in \cite{Filipovic-Mayerhofer}, where existence of affine processes has been proven in finite dimension. After this step, we obtain the existence of weak solutions by using a refined version of a result from \cite{GMR}, where $H_{I,0}$ serves as the retracted subspace with compact embedding, and pathwise uniqueness follows from a version of the uniqueness result from \cite{Yamada-Watanabe-1971} in infinite dimension. 

We start with characterizations when the drift is inward pointing, and when the volatility is parallel.

\begin{proposition}\label{prop-inward}
The following statements are equivalent:
\begin{enumerate}
\item[(i)] The mapping $\mu$ is inward pointing at boundary points of $\calx$.

\item[(ii)] We have
\begin{align*}
\la \mu(x),e_i \ra_H \geq 0 \quad \text{for all $x \in \calx$ and all $i \in I$ with $\la x,e_i \ra_H = 0$.}
\end{align*}

\item[(iii)] We have
\begin{align}\label{cond-inward-1}
m_0 &\in \calx,
\\ \label{cond-inward-2} M x &\in ( H_I^+ + \lin \{ e_i \} ) \oplus H_J \quad \text{for all $i \in I$ and $x \in \lin^+ \{ e_i \}$,}
\\ \label{cond-inward-3} M(H_J) &\subset H_J.
\end{align}
\end{enumerate}
\end{proposition}

Suppose that $\mu$ is inward pointing at boundary points of $\calx$. Then \eqref{cond-inward-1} immediately yields that $m_{0,I} \in H_I^+$, and therefore, the condition $m_{0,I} \in H_{I,0}^+$ appearing in (\ref{ass-prin-drift}) is equivalent to $m_{0,I} \in H_{I,0}$.

\begin{proof}[Proof of Proposition \ref{prop-inward}]
(i) $\Leftrightarrow$ (ii): This equivalence is straightforward to check.

\noindent(ii) $\Leftrightarrow$ (iii): The proof of this equivalence is analogous to that of \cite[Prop. A.10]{Tappe-affin}.
\end{proof}

\begin{proposition}\label{prop-parallel}
The following statements are equivalent:
\begin{enumerate}
\item[(i)] The mapping $\sigma$ is parallel to the boundary at boundary points of $\calx$.

\item[(ii)] We have
\begin{align*}
\la \sigma(x),e_i \ra_H = 0 \quad \text{for all $x \in \calx$ and all $i \in I$ with $\la x,e_i \ra_H = 0$.}
\end{align*}

\item[(iii)] We have
\begin{align}\label{cond-parallel-1}
n_0 \xi &= 0 \quad \text{for all $\xi \in H_I$,}
\\ \label{cond-parallel-2} N(x) &= 0 \quad \text{for all $x \in H_J$,}
\\ \label{cond-parallel-3} N(x)\xi &= 0 \quad \text{for all $i,j \in I$ with $i \neq j$ and all $x \in \lin^+ \{ e_i \}$ and $\xi \in \lin^+ \{ e_j \}$.}
\end{align}
\end{enumerate}
\end{proposition}

\begin{proof}
(i) $\Leftrightarrow$ (ii): This equivalence is straightforward to check.

\noindent(ii) $\Leftrightarrow$ (iii): Note that (\ref{sigma-par-def}) is satisfied if and only if
\begin{align*}
\sigma(x)^* \eta = 0 \quad \text{for all $x \in \calx$ and all $\eta \in H_I^+$ with $\la x,\eta \ra_H = 0$.}
\end{align*}
Since $(\Sigma_W^{\nicefrac{1}{2}})^*$ is one-to-one, by (\ref{S-sigma}) this is equivalent to
\begin{align*}
\la S(x)\eta, \eta \ra_H = 0 \quad \text{for all $x \in \calx$ and all $\eta \in H_I^+$ with $\la x,\eta \ra_H = 0$.}
\end{align*}
Therefore, the proof of this equivalence is analogous to that of \cite[Prop. A.20]{Tappe-affin}.
\end{proof}

Consequently, the inward pointing property of $\mu$ and the parallel property of $\sigma$ mean that the parameters $m_0,M,n_0,N$ satisfy the admissibility conditions from Proposition \ref{prop2.5}. Hence, in this case the general Riccati system \eqref{Riccati1}-\eqref{Riccati2} has a unique solution $(\phi(\cdot,u),\psi(\cdot,u)) : \R_+ \to \bbC_- \times (H_\bbC^-)_I \oplus iH _J$ for each
$u \in (H_\bbC^-)_I \oplus i H_J$.

From now on, we assume that $\mu$ is inward pointing, and that $\sigma$ is parallel. The following result in particular shows that $\lambda \in \ell^2(I)$, and that the dispersion operator $S$ restricted to $H_I$ has a diagonal structure.

\begin{proposition}\label{prop-lambda}
The following statements are true:
\begin{enumerate}[(i)]
\item We have $S(x) = S(x_I)$ for all $x \in \calx$.

\item We have $S(x)\xi = N(x)\xi$ for all $x \in H_I^+$ and $\xi \in H_I$.

\item We have $S(e_i)e_j = 0$ for all $i,j \in I$ with $i \neq j$.

\item We have $\lambda \in \ell^2(I)$ and the representation
\begin{align}\label{lambda-repr-main-part}
\lambda_i = \la S(e_i)_{II} \, e_i, e_i \ra_H, \quad i \in I. 
\end{align}

\item We have the representation
\begin{align}\label{op-repr-general-main-part}
S(x)_{II} \, \xi = \sum_{i \in I} \lambda_i \la x,e_i \ra_H \la e_i,\xi \ra_H \, e_i \quad \text{for all $x \in H_I^+$ and $\xi \in H_I$.}
\end{align}

\end{enumerate}
\end{proposition}

\begin{proof}
By condition (\ref{cond-parallel-2}) from Proposition \ref{prop-parallel} we have $S(x) = S(x_I)$ for all $x \in \calx$, and by condition (\ref{cond-parallel-1}) from Proposition \ref{prop-parallel} we have $S(x)\xi = N(x)\xi$ for all $x \in H_I^+$ and $\xi \in H_I$. Therefore, by condition (\ref{cond-parallel-3}) from Proposition \ref{prop-parallel} we have $S(e_i)e_j = 0$ for all $i,j \in I$ with $i \neq j$. Now, let $i \in I$ be arbitrary. Note that $S(e_i)_{II}$ is self-adjoint, because $\pi_I S(e_i) \pi_I$ is self-adjoint. Thus, we obtain
\begin{align*}
\la S(e_i)_{II}\,e_i,e_j \ra_H = \la e_i,S(e_i)_{II}\,e_j \ra_H = 0 \quad \text{for each $j \in I$ with $j \neq i$,}
\end{align*}
and hence $S(e_i)_{II} \, e_i \in \lin \{ e_i \}$. Therefore, we have
\begin{align*}
S(e_i)_{II} \, e_i = \la S(e_i)_{II}e_i,e_i \ra_H \, e_i,
\end{align*}
and hence, noting that $S(e_i)_{II} \in L_1^+(H_I)$, we obtain
\begin{align*}
\lambda_i = \| S(e_i)_{II} \, e_i \|_H = \la S(e_i)_{II} \, e_i, e_i \ra_H,
\end{align*}
showing (\ref{lambda-repr-main-part}) and $S(e_i)_{II} \, e_i = \lambda_i e_i$, which also proves (\ref{op-repr-general-main-part}). Now, let $\nu \in \ell^2(I)$ be arbitrary, and set $y := \sum_{i \in I} \nu_i e_i \in H_I$. Then the series
\begin{align*}
\sum_{i \in I} \lambda_i \nu_i = \sum_{i \in I} \la S(e_i)_{II} \, e_i, e_i \ra_H \, \nu_i = \sum_{i \in I} \la S(y)_{II} \, e_i, e_i \ra_H
\end{align*}
converges, because $S(y)_{II} \in L_1(H_I)$. By the uniform boundedness principle we deduce that $\lambda \in \ell^2(I)$.
\end{proof}

Now, we will deal with linear transformations which leave the state space $\calx$ invariant. The next result provides a characterization of such transformations. 

\begin{lemma}\label{lemma-X-inv-trans}
For a bounded linear operator $\Lambda \in L(H)$ the following statements are equivalent:
\begin{enumerate}
\item[(i)] We have $\Lambda(\calx) \subset \calx$.

\item[(ii)] We have $\Lambda^*(H_I^+) \subset H_I^+$ and $\Lambda(H_J) \subset H_J$.
\end{enumerate}
\end{lemma}

\begin{proof}
(i) $\Rightarrow$ (ii): Suppose there exists $x \in H_J$ with $\Lambda x \in \calx \setminus H_J$. Since $\Lambda x \notin H_J$, there exists $i \in I$ with $\la \Lambda x,e_i \ra_H > 0$. We have $-x \in H_J \subset \calx$, and hence 
\begin{align*}
\la \Lambda(-x),e_i \ra_H = - \la \Lambda x,e_i \ra_H < 0,
\end{align*}
which provides the contradiction $\Lambda(-x) \notin \calx$. Therefore, we have $\Lambda(H_J) \subset H_J$, and hence
\begin{align*}
\la \Lambda^* x,y \ra_H = \la x,\Lambda y \ra_H = 0 \quad \text{for all $x \in H_I$ and all $y \in H_J$,}
\end{align*}
which shows $\Lambda^*(H_I) \subset H_I$. Furthermore, we have
\begin{align*}
\la \Lambda^* x,y \ra_H = \la x,\Lambda y \ra_H \geq 0 \quad \text{for all $x,y \in H_I^+$,}
\end{align*}
showing that $\Lambda^*(H_I^+) \subset H_I^+$.

\noindent (ii) $\Rightarrow$ (i): For all $x,y \in H_I^+$ we have
\begin{align*}
\la \Lambda x,y \ra_H = \la x,\Lambda^* y \ra_H \geq 0,
\end{align*}
and hence we deduce $\Lambda(H_I^+) \subset \calx$. Therefore, for each $x \in \calx$ we obtain
\begin{align*}
\Lambda x = \Lambda x_I + \Lambda x_J \in \calx,
\end{align*}
completing the proof.
\end{proof}

Now, let $\Lambda \in L(H)$ be an isomorphism such that $\Lambda(\calx) = \calx$. We introduce the new mappings $\bar{\mu} : \calx \to H$ and $\bar{\sigma} : \calx \to L_2^+(U_0,H)$ as
\begin{align}\label{def-mu-bar}
\bar{\mu}(y) &:= \Lambda \mu(x), \quad y \in \calx,
\\ \label{def-sigma-bar} \bar{\sigma}(y) &:= \Lambda \sigma(x), \quad y \in \calx,
\end{align}
where $x = \Lambda^{-1}y \in \calx$, and we define the new mapping $\bar{S} : \calx \to L_1^+(H)$ as
\begin{align}\label{def-S-bar} 
\bar{S}(y) &:= \bar{\sigma}(y) \Sigma_W^{\nicefrac{1}{2}} \big( \bar{\sigma}(y) \Sigma_W^{\nicefrac{1}{2}} \big)^*, \quad y \in \calx.
\end{align}
Taking into account (\ref{S-sigma}), it is easy to check that
\begin{align}\label{bar-S-Lambda}
\bar{S}(y) = \Lambda S(x) \Lambda^* = \Lambda S(x)^{\nicefrac{1}{2}} \big( \Lambda S(x)^{\nicefrac{1}{2}} \big)^* \quad \text{for all $y \in \calx$,}
\end{align}
where $x = \Lambda^{-1}y \in \calx$. Note that for a solution $X$ to the SDE (\ref{SDE-affine}) the process $Y := \Lambda X$ is a solution to the SDE (\ref{SDE-affine-Y}) with $y_0 = \Lambda x_0$. The upcoming results show that all relevant properties are still satisfied for the new parameters.

\begin{lemma}\label{lemma-inward-preserved}
The following statements are true:
\begin{enumerate}[(i)]
\item The mapping $\bar{\mu}$ is inward pointing at boundary points of $\calx$.

\item The mapping $\bar{\sigma}$ is parallel to the boundary at boundary points of $\calx$.
\end{enumerate}
\end{lemma}

\begin{proof}
Let $y \in \calx$ and $\eta \in H_I^+$ with $\la y,\eta \ra_H = 0$ be arbitrary. We set $x := \Lambda^{-1} y \in \calx$. By Lemma \ref{lemma-X-inv-trans} we have $\Lambda^* \eta \in H_I^+$. Furthermore, we have
\begin{align*}
\la x, \Lambda^* \eta \ra_H = \la \Lambda x, \eta \ra_H = \la y,\eta \ra_H = 0.
\end{align*}
Therefore, if $\mu$ is inward pointing, then we obtain
\begin{align*}
\la \bar{\mu}(y), \eta \ra_H = \la \Lambda(\mu(x)), \eta \ra_H = \la \mu(x), \Lambda^* \eta \ra_H \geq 0.
\end{align*}
Similarly, if $\sigma$ is parallel, then we obtain
\begin{align*}
\la \bar{\sigma}(y), \eta \ra_H = \la \Lambda(\sigma(x)), \eta \ra_H = \la \sigma(x), \Lambda^* \eta \ra_H = 0,
\end{align*}
finishing the proof.
\end{proof}

Now, we define $\bar{m}_0 \in H$ and $\bar{M} \in L(H)$ as
\begin{align}\label{def-m-bar}
\bar{m}_0 := \Lambda m_0 \quad \text{and} \quad \bar{M} := \Lambda M \Lambda^{-1}.
\end{align}
Then, using (\ref{def-mu-bar}), (\ref{mu-affine}) and (\ref{def-m-bar}) it is easy to check that $\bar{\mu}$ has the affine structure
\begin{align}\label{mu-bar-affine}
\bar{\mu}(y) = \bar{m}_0 + \bar{M}y \quad \text{for all $y \in \calx$.}
\end{align}
Let us decompose $\bar{\mu}$ with respect to $H = H_I \oplus H_J$. We define the affine mappings $\bar{\mu}_{II} : H_I^+ \to H_I$ and $\bar{\mu}_J : \calx \to H_J$ as
\begin{align}\label{def-mu-bar-I}
\bar{\mu}_{II}(y) &:= \bar{m}_{0,I} + \bar{M}_{II} y, \quad y \in H_I^+,
\\ \label{def-mu-bar-J} \bar{\mu}_J(y) &:= \bar{m}_{0,J} + \bar{M}_J y, \quad y \in \calx.
\end{align}
Then we have the decomposition (\ref{mu-decomp}), which easily follows from Lemma \ref{lemma-inward-preserved}, condition (\ref{cond-inward-3}), and since $\pi_I \bar{M} \pi_J = 0$ according to Proposition \ref{prop-inward}. The next result shows that the inward pointing property also transfers to $\bar{\mu}_{II}$.

\begin{lemma}\label{lemma-mII-inward}
The mapping $\bar{\mu}_{II}$ is inward pointing at boundary points of $H_I^+$.
\end{lemma}

\begin{proof}
Taking into account (\ref{mu-bar-affine}), by Proposition \ref{prop-inward} we have
\begin{align*}
\bar{m}_0 &\in \calx,
\\ \bar{M} x &\in ( H_I^+ + \lin \{ e_i \} ) \oplus H_J \quad \text{for all $i \in I$ and $x \in \lin^+ \{ e_i \}$,}
\\ \bar{M}(H_J) &\subset H_J.
\end{align*}
Therefore, we have
\begin{align*}
\bar{m}_{0,I} &\in H_I^+,
\\ \bar{M}_{II} x &\in ( H_I^+ + \lin \{ e_i \} ) \quad \text{for all $i \in I$ and $x \in \lin^+ \{ e_i \}$.}
\end{align*}
Hence, taking into account (\ref{def-mu-bar-I}), by Proposition \ref{prop-inward} we deduce that $\bar{\mu}_{II}$ is inward pointing at boundary points of $H_I^+$.
\end{proof}

So far, we have considered a general transformation $\Lambda$, which leaves the state space $\calx$ invariant. Now, we will consider a concrete choice for this transformation, which will provide the announced block diagonal structure (\ref{sigma-decomp}) of the volatility $\bar{\sigma}$. By (\ref{frac-in-l2}) and the Cauchy-Schwarz inequality, the mapping
\begin{align}\label{D-define}
Dx := - \sum_{i \in I_{>0}} \la x,e_i \ra_H \, \frac{S(e_i)_{IJ} \, e_i}{\lambda_i}, \quad x \in H
\end{align}
is a well-defined continuous linear operator $D \in L(H)$. We define $\Lambda \in L(H)$ as 
\begin{align}\label{Lambda-def}
\Lambda := \Id + D.
\end{align}
We require some auxiliary results, before we can prove the block diagonal structure (\ref{sigma-decomp}) of $\bar{\sigma}$ in Proposition \ref{prop-sigma} later on.

\begin{lemma}\label{lemma-Lambda-properties}
The following statements are true:
\begin{enumerate}[(i)]
\item We have $\ran(D) \subset H_J \subset \ker(D)$.

\item We have $\ran(D^*) \subset H_I \subset \ker(D^*)$.

\item $\Lambda$ is an isomorphism with $\Lambda^{-1} = \Id - D$.

\item We have $\Lambda(\calx) = \calx$.

\item We have $\Lambda x = x$ and $\Lambda^{-1}x = x$ for all $x \in H_J$.
\end{enumerate}
\end{lemma}

\begin{proof}
The first statement immediately follows from (\ref{D-define}), and since $H_J^{\perp} = H_I$, we obtain
\begin{align*}
\ran(D^*) \subset \overline{\ran(D^*)} = (\ran(D^*)^{\perp})^{\perp} = \ker(D)^{\perp} \subset H_I \subset \ran(D)^{\perp} = \ker(D)^*.
\end{align*}
Since $\ran(D) \subset \ker(D)$, we also have $D^2 = 0$, which gives us
\begin{align*}
(\Id + D)(\Id - D) = \Id - D^2 = \Id,
\end{align*}
showing that $\Lambda$ is an isomorphism with $\Lambda^{-1} = \Id - D$. Furthermore, taking into account $\ran(D) \subset H_J$, we obtain $\Lambda(\calx) \subset \calx$ and $\Lambda^{-1}(\calx) \subset \calx$, and hence $\Lambda(\calx) = \calx$. Finally, since $H_J \subset \ker(D)$, we have $Dx = 0$ for all $x \in H_J$, and hence $\Lambda x = x$ and $\Lambda^{-1}x = x$ for all $x \in H_J$.
\end{proof}

The following auxiliary result concerns the quantities appearing in the affine structures (\ref{mu-affine}) and (\ref{mu-bar-affine}) of $\mu$ and $\bar{\mu}$.

\begin{lemma}\label{lemma-ass-3}
We have $m_{0,I} = \bar{m}_{0,I}$ and $M_{II} = \bar{M}_{II}$.
\end{lemma}

\begin{proof}
By (\ref{def-m-bar}) and Lemma \ref{lemma-Lambda-properties} we have
\begin{align*}
\bar{m}_{0,I} = \pi_I \bar{m}_0 = \pi_I \Lambda m_0 = \pi_I (\Id + D) m_0 = \pi_I m_0 + \pi_I D m_0 = m_{0,I}.
\end{align*}
Furthermore, by (\ref{def-m-bar}) and Lemma \ref{lemma-Lambda-properties} we have
\begin{align*}
\bar{M} = \Lambda M \Lambda^{-1} = (\Id + D) M (\Id - D) = M + DM - MD - DMD.
\end{align*}
By Proposition \ref{prop-inward} we have $M(H_J) \subset H_J$. Therefore, by Lemma \ref{lemma-Lambda-properties} we obtain
\begin{align*}
(DM)_{II} = (MD)_{II} = (DMD)_{II} = 0,
\end{align*}
and hence $M_{II} = \bar{M}_{II}$.
\end{proof}

Now, we consider the dispersion operator $S$. By Proposition \ref{prop-lambda} we have $\la S(e_i)e_i,e_i \ra_H = 0$, and hence
\begin{align}\label{rel-I0}
S(e_i)e_i = 0 \quad \text{for each $i \in I_0 := I \setminus I_{>0}$.}
\end{align}

\begin{lemma}\label{lemma-D-equations}
The following statements are true:
\begin{enumerate}[(i)]
\item We have $D S(x)_I \, \pi_I = - S(x)_J \, \pi_I$ for all $x \in H_I^+$.

\item We have $S(x)_I \, D^* = - S(x)_I \, \pi_J$ for all $x \in H_I^+$.
\end{enumerate}
\end{lemma}

\begin{proof}
By (\ref{D-define}) we have
\begin{align*}
D e_i = - \frac{S(e_i)_J \, e_i}{\lambda_i} \quad \text{for all $i \in I_{>0}$.}
\end{align*}
Therefore, by Proposition \ref{prop-lambda} and (\ref{rel-I0}), for all $x \in H_I^+$ and $\xi \in H_I$ we obtain
\begin{align*}
&D S(x)_I \, \xi = D \bigg( \sum_{i \in I_{>0}} \lambda_i \la x,e_i \ra_H \la e_i,\xi \ra_H \, e_i \bigg) = -\sum_{i \in I_{>0}} \la x,e_i \ra_H \la e_i,\xi \ra_H S(e_i)_J \, e_i
\\ &= -\sum_{i \in I} \la x,e_i \ra_H \la e_i,\xi \ra_H S(e_i)_J \, e_i = - S \bigg( \sum_{i \in I} \la x,e_i \ra_H \, e_i \bigg)_J \, \sum_{i \in I} \la \xi,e_i \ra_H \, e_i = -S(x)_J \, \xi.
\end{align*}
Therefore, we have
\begin{align*}
D \pi_I S(x) \pi_I = - \pi_J S(x) \pi_I.
\end{align*}
By Lemma \ref{lemma-Lambda-properties} we have $\ran(D^*) \subset H_I$. Therefore, taking adjoints we obtain
\begin{align*}
S(x)_I \, D^* = \pi_I S(x) \pi_I D^* = (D \pi_I S(x) \pi_I)^* = - \pi_I S(x) \pi_J = -S(x)_I \, \pi_J,
\end{align*}
completing the proof.
\end{proof}

Now, we introduce $\bar{n}_0 \in L_1^+(H)$ and $\bar{N} \in L(H,L_1(H))$ as
\begin{align}\label{def-n-bar}
\bar{n}_0 := \Lambda n_0 \Lambda^* \quad \text{and} \quad \bar{N}y := \Lambda Nx \Lambda^* \quad \text{for all $y \in H$,}
\end{align}
where $x = \Lambda^{-1} y \in H$. Then $\bar{S}$ has the affine structure
\begin{align}\label{S-bar-affine}
\bar{S}(y) = \bar{n}_0 + \bar{N} y \quad \text{for all $y \in \calx$,}
\end{align}
which is easily checked by using (\ref{bar-S-Lambda}), (\ref{S-affine}) and (\ref{def-n-bar}).

\begin{lemma}\label{lemma-D-act}
For each $y \in \calx$ the following statements are true:
\begin{enumerate}[(i)]
\item We have $\bar{S}(y) \pi_I = S(x)_I \, \pi_I$, where $x = \Lambda^{-1} y \in \calx$, and hence $\bar{S}(y)(H_I) \subset H_I$.

\item We have $\bar{S}(y)(H_J) \subset H_J$.
\end{enumerate}
\end{lemma}

\begin{proof}
Let $y \in \calx$ be arbitrary, and set $x := \Lambda^{-1} y \in \calx$. Furthermore, let $\eta \in H_I$ be arbitrary. By Lemma \ref{lemma-Lambda-properties} we have $D^* \eta = 0$. Therefore, by (\ref{bar-S-Lambda}) and Lemma \ref{lemma-D-equations} we obtain
\begin{align*}
\bar{S}(y)\eta = \Lambda S(x) \Lambda^* \eta &= (\Id + D) S(x) (\eta + D^* \eta)
\\ &= (\Id + D)S(x)\eta = (\Id + D)(S(x)_I \, \eta + S(x)_J \, \eta)
\\ &= (\Id + D)S(x)_I \, \eta + S(x)_J \, \eta 
\\ &= S(x)_I \, \eta + D S(x)_I \, \eta + S(x)_J \, \eta = S(x)_I \, \eta.
\end{align*}
Now, let $\eta \in H_J$ be arbitrary. Then by (\ref{bar-S-Lambda}) we have
\begin{align*}
\bar{S}(y)\eta = \Lambda S(x) \Lambda^* \eta &= (\Id + D) S(x) (\eta + D^* \eta)
\\ &= (\Id + D) S(x)_I (\eta + D^* \eta) + S(x)_J(\eta + D^* \eta)
\\ &= S(x)_I \, \eta + S(x)_I \, D^* \eta + D S(x)_I (\eta + D^* \eta) + S(x)_J (\eta + D^* \eta).
\end{align*}
Note that by Lemma \ref{lemma-Lambda-properties} we have
\begin{align*}
D S(x)_I(\eta + D^* \eta) + S(x)_J (\eta + D^* \eta) \eta \in H_J.
\end{align*}
Furthermore, by Lemma \ref{lemma-D-equations} we obtain
\begin{align*}
S(x)_I \, \eta + S(x)_I \, D^* \eta = 0,
\end{align*}
and hence $\bar{S}(y)\eta \in H_J$, completing the proof.
\end{proof}

Now, we are ready to analyze the structure of $\bar{S}$ and its square root. As the next result shows, $\bar{S}$ has a block diagonal structure and $\bar{S}_{II}$ has a diagonal structure.\\[2mm]

\begin{proposition}\label{prop-block-diagonal}
We have that 
\begin{enumerate}[(i)]
\item for all $y \in \calx$ and $\eta \in H$,
\begin{align}\label{S-block-1}
\bar{S}(y)\eta = \bar{S}(y_I)_{II} \, \eta_I + \bar{S}(y_I)_{JJ} \, \eta_J,
\end{align}

\item for all $y \in H_I^+$ and $\eta \in H_I$,
\begin{align}\label{S-block-2}
\bar{S}(y)_{II} \, \eta = \sum_{i \in I} \lambda_i \la y,e_i \ra_H \la e_i,\eta \ra_H \, e_i.
\end{align}
\end{enumerate}
\end{proposition}

\begin{proof}
The mapping $\bar{S}$ has the affine structure (\ref{S-bar-affine}), and by Lemma \ref{lemma-inward-preserved} the mapping $\bar{\sigma}$ is parallel to the boundary at boundary points of $\calx$. Therefore, by Proposition \ref{prop-lambda} we have
\begin{align*}
\bar{S}(y) = \bar{S}(y_I) \quad \text{for all $y \in \calx$.}
\end{align*}
Now, let $y \in \calx$ and $\eta \in H$ be arbitrary. By Lemma \ref{lemma-D-act} we have
\begin{align*}
\bar{S}(y)\eta = \bar{S}(y_I) \eta = \bar{S}(y_I) \eta_I + \bar{S}(y_I) \eta_J = \bar{S}(y_I)_{II} \, \eta_I + \bar{S}(y_I)_{JJ} \, \eta_J.
\end{align*}
Now, let $y \in H_I^+$ be arbitrary. By Lemma \ref{lemma-Lambda-properties}, for each $i \in I$ we have
\begin{align*}
\la \Lambda^{-1} y, e_i \ra_H = \la y - Dy, e_i \ra_H = \la y,e_i \ra_H - \la y, D^* e_i \ra = \la y,e_i \ra_H.
\end{align*}
Therefore, by Lemma \ref{lemma-D-act} and Proposition \ref{prop-lambda}, for all $y \in H_I^+$ and $\eta \in H_I$ we obtain
\begin{align*}
\bar{S}(y)_{II} \, \eta = S(\Lambda^{-1} y)_{II} \, \eta = \sum_{i \in I} \lambda_i \la \Lambda^{-1}y,e_i \ra_H \la e_i,\eta \ra_H \, e_i = \sum_{i \in I} \lambda_i \la y,e_i \ra_H \la e_i,\eta \ra_H \, e_i,
\end{align*}
completing the proof.
\end{proof}

Let us analyze the structure of the volatility $\bar{\sigma}$. Recall that $U=H$, that the covariance operator $\Sigma_W \in L_1^{++}(U)$ has a diagonal structure along $( e_k )_{k \in \bbn}$, and that for each $x \in \calx$ the operator $\sigma(x) \Sigma_W^{\nicefrac{1}{2}}$ is self-adjoint. Therefore, we have
\begin{align}\label{sigma-bar-S-bar}
\bar{\sigma}(y) = \bar{S}(y)^{\nicefrac{1}{2}} \Sigma_W^{-\nicefrac{1}{2}} \quad \text{for all $y \in \calx$,}
\end{align}
which easily follows from (\ref{bar-S-Lambda}) and (\ref{sigma-S}). Recall that $U_0 = \Sigma_W^{\nicefrac{1}{2}}(U)$ is a separable Hilbert space with inner product
\begin{align}\label{inner-product-U0}
\la u,v \ra_{U_0} := \big\la \Sigma_W^{-\nicefrac{1}{2}} u, \Sigma_W^{-\nicefrac{1}{2}} v \big\ra_U, \quad u,v \in U_0.
\end{align}
The system $( g_k )_{k \in \bbn}$ given by
\begin{align}\label{ONB-gk}
g_k = \Sigma_W^{\nicefrac{1}{2}} e_k, \quad k \in \bbn
\end{align}
is an orthonormal basis of $U_0$. Now,  define
\begin{align*}
U_{I,0} := \Sigma_W^{\nicefrac{1}{2}}(U_I) \subset U_I \quad \text{, } \quad U_{J,0} := \Sigma_W^{\nicefrac{1}{2}}(U_J) \subset U_J,
\end{align*}
and define the mappings $\bar{\sigma}_{II} : H_I^+ \to L_2(U_{I,0},H_I)$ and $\bar{\sigma}_{JJ} : H_I^+ \to L_2(U_{J,0},H_J)$ as
\begin{align}\label{def-bar-sigma-I}
\bar{\sigma}_{II}(y) &:= \bar{S}(y)_{II}^{\nicefrac{1}{2}} \, \Sigma_{W,II}^{-\nicefrac{1}{2}}, \quad y \in H_I^+,
\\ \label{def-bar-sigma-J} \bar{\sigma}_{JJ}(y) &:= \bar{S}(y)_{JJ}^{\nicefrac{1}{2}} \, \Sigma_{W,JJ}^{-\nicefrac{1}{2}}, \quad y \in H_I^+.
\end{align}
Mainly as a consequence of the block diagonal structure of $\bar{S}$, we can now prove the announced block diagonal structure of $\bar{\sigma}$. In addition, we obtain that $\bar{\sigma}_{II}$ has a diagonal structure.

\begin{proposition}\label{prop-sigma}
The following statements are true:
\begin{enumerate}[(i)]
\item We have the representation (\ref{sigma-decomp}).

\item We have the representation
\begin{align}\label{repr-sigma-block}
\bar{\sigma}_{II}(y)u = \sum_{i \in I} \sqrt{\lambda_i \la y,e_i \ra_H } \, \la g_i, u \ra_{U_0} \, e_i \quad \text{for all $y \in H_I^+$ and $u \in U_{I,0}$.}
\end{align}

\item If $\lambda \in \ell^1(I)$ and $g_i = \sqrt{\lambda_i} e_i$ for all $i \in I$, then we have the representation
\begin{align*}
\bar{\sigma}_{II}(y)u = \sum_{i \in I} \sqrt{\la y,e_i \ra_H } \, \la e_i, u \ra_U \, e_i \quad \text{for all $y \in H_I^+$ and $u \in U_{I,0}$.}
\end{align*}

\end{enumerate}
\end{proposition}

\begin{proof}
Let $u \in U_0$ be arbitrary. Note that
\begin{align*}
\Sigma_W^{-\nicefrac{1}{2}} u = \Sigma_{W,II}^{-\nicefrac{1}{2}} \, u_I + \Sigma_{W,JJ}^{-\nicefrac{1}{2}} \, u_J
\end{align*}
is the decomposition of $\Sigma_W^{-\nicefrac{1}{2}} u$ according to $U = U_I \oplus U_J$. Therefore, by (\ref{sigma-bar-S-bar}) and Proposition \ref{prop-block-diagonal}, for each $y \in \calx$ we have
\begin{align*}
\bar{\sigma}(y)u &= \bar{S}(y)^{\nicefrac{1}{2}} \Sigma_W^{-\nicefrac{1}{2}} u
= \bar{S}(y_I)_{II}^{\nicefrac{1}{2}} \, \Sigma_{W,II}^{-\nicefrac{1}{2}} \, u_I + \bar{S}(y_I)_{JJ}^{\nicefrac{1}{2}} \, \Sigma_{W,JJ}^{-\nicefrac{1}{2}} \, u_J 
\\ &= \bar{\sigma}_{II}(y_I)u_I + \bar{\sigma}_{JJ}(y_I)u_J.
\end{align*}
Furthermore, by Proposition \ref{prop-block-diagonal} and (\ref{inner-product-U0}), for all $y \in H_I^+$ and $u \in U_{I,0}$ we have
\begin{align*}
\bar{\sigma}_{II}(y)u &= \bar{S}(y)_{II}^{\nicefrac{1}{2}} \, \Sigma_{W,II}^{-\nicefrac{1}{2}} \, u = \sum_{i \in I} \sqrt{\lambda_i \la y,e_i \ra_H} \, \la e_i, \Sigma_{W,II}^{-\nicefrac{1}{2}} \, u \ra_H \, e_i
\\ &= \sum_{i \in I} \sqrt{\lambda_i \la y,e_i \ra_H } \, \la \Sigma_{W,II}^{\nicefrac{1}{2}} e_i, u \ra_{U_0} \, e_i = \sum_{i \in I} \sqrt{\lambda_i \la y,e_i \ra_H} \, \la g_i, u \ra_{U_0} \, e_i.
\end{align*}
If $\lambda \in \ell^1(I)$ and $g_i = \sqrt{\lambda_i} e_i$ for all $i \in I$, then by (\ref{inner-product-U0}) we obtain
\begin{align*}
\bar{\sigma}_{II}(y)u &= \sum_{i \in I} \sqrt{\lambda_i \la y,e_i \ra_H} \, \la g_i, u \ra_{U_0} \, e_i = \sum_{i \in I} \lambda_i \sqrt{\la y,e_i \ra_H} \, \la e_i, u \ra_{U_0} \, e_i
\\ &= \sum_{i \in I} \lambda_i \sqrt{\la y,e_i \ra_H } \, \la \Sigma_{W,II}^{-\nicefrac{1}{2}} \, e_i, \Sigma_{W,II}^{-\nicefrac{1}{2}} \, u \ra_U \, e_i = \sum_{i \in I} \sqrt{\la y,e_i \ra_H } \, \la e_i, u \ra_U \, e_i,
\end{align*}
completing the proof.
\end{proof}

By virtue of the decompositions (\ref{mu-decomp}) and (\ref{sigma-decomp}) of $\bar{\mu}$ and $\bar{\sigma}$ we can express the transformed SDE (\ref{SDE-affine-Y}) by the two coupled SDEs (\ref{SDE-affine-Y-I}) and (\ref{SDE-affine-Y-J}). Hence our task is essentially reduced to solve the SDE (\ref{SDE-affine-Y-I}). The aim of the following auxiliary results is to show that all conditions for an application of our existence result for weak solutions (Theorem \ref{thm-SDE-existence}) are fulfilled. Recall that by virtue of Lemma \ref{lemma-mII-inward} the drift $\bar{\mu}_{II}$ is inward pointing at boundary points of $H_I^+$. Furthermore, the volatility $\bar{\sigma}_{II}$ is parallel to the boundary at boundary points of $H_I^+$, which is a consequence of the diagonal structure (\ref{repr-sigma-block}) from Proposition \ref{prop-sigma}. Moreover, by (\ref{def-mu-bar-I}) the drift $\bar{\mu}_{II} : H_I^+ \to H_I$ satisfies the linear growth condition and is Lipschitz continuous.

\begin{lemma}\label{lemma-lin-growth-sigma}
The mapping $\bar{\sigma}_{II} : H_I^+ \to L_2(U_{I,0},H_I)$ satisfies the linear growth condition, and for all $y,z \in H_I^+$ we have
\begin{align*}
\| \bar{\sigma}_{II}(y) - \bar{\sigma}_{II}(z) \|_{L_2(U_{I,0},H_I)}^2 \leq \| \lambda \|_{\ell^2(I)} \| y - z \|_H.
\end{align*}
\end{lemma}

\begin{proof}
The mapping $\bar{S}$ has the affine structure (\ref{S-bar-affine}), and by Lemma \ref{lemma-inward-preserved} the mapping $\bar{\sigma}$ is parallel to the boundary at boundary points of $\calx$. Therefore, by Proposition \ref{prop-lambda} we have
\begin{align*}
\bar{S}(y)\eta = \bar{N}(y) \eta \quad \text{for all $y \in H_I^+$ and $\eta \in H_I$.}
\end{align*}
Since $(g_i)_{i \in I}$ is an orthonormal basis of $U_{I,0}$, by (\ref{def-bar-sigma-I}) for each $y \in H_I^+$ we obtain
\begin{align*}
\| \bar{\sigma}_{II}(y) \|_{L_2(U_{I,0},H_I)}^2 &= \sum_{i \in I} \| \bar{\sigma}_{II}(y) g_i \|_{H}^2 = \sum_{i \in I} \| \bar{\sigma}_{II}(y) \Sigma_{W,II}^{\nicefrac{1}{2}} e_i \|_{H}^2 = \big\| \bar{\sigma}_{II}(y) \Sigma_{W,II}^{\nicefrac{1}{2}} \big\|_{L_2(H_I)}^2 
\\ &= \big\| \bar{S}(y)_{II}^{\nicefrac{1}{2}} \big\|_{L_2(H_I)}^2 = \| \bar{S}(y)_{II} \|_{L_1(H_I)} \leq \| \bar{N}_{II} \|_{L(H_I,L_1(H_I))} \| y \|_H,
\end{align*}
where the operator $\bar{N}_{II} \in L(H_I,L_1(H_I))$ is given by $\bar{N}_{II} \eta = (\bar{N} \eta)_{II}$ for $\eta \in H_I$. This proves that $\bar{\sigma}_{II}$ satisfies the linear growth condition. Furthermore, by Proposition \ref{prop-block-diagonal} for all $y,z \in H_I^+$ we have
\begin{align*}
&\| \bar{\sigma}_{II}(y) - \bar{\sigma}_{II}(z) \|_{L_2(U_{I,0},H_I)}^2 = \big\| ( \bar{\sigma}_{II}(y) - \bar{\sigma}_{II}(z) ) \Sigma_{W,II}^{\nicefrac{1}{2}} \big\|_{L_2(H_I)}^2
\\ &= \big\| \bar{S}(y)_{II}^{\nicefrac{1}{2}} - \bar{S}(z)_{II}^{\nicefrac{1}{2}} \big\|_{L_2(H_I)}^2 = \sum_{i \in I} \Big( \sqrt{\lambda_i \la y,e_i \ra_H} - \sqrt{\lambda_i \la z,e_i \ra_H} \Big)^2
\\ &\leq \sum_{i \in I} | \lambda_i \la y - z,e_i \ra_H | \leq \| \lambda \|_{\ell^2(I)} \| y - z \|_H,
\end{align*}
finishing the proof.
\end{proof}
  
We now define the retracted subspace with compact embedding. Since $\lambda \in \ell^2(I)$, there exists a sequence $\nu = (\nu_i)_{i \in I} \subset (0,\infty)$ with $\nu_i \to 0$ such that $( \lambda_i / \nu_i )_{i \in I} \in \ell^2(I)$, which gives rise to the compact linear operator $T \in K^{++}(H_I)$ with representation (\ref{T-compact-repr}) and the retracted subspace $H_{I,0} := T(H_I)$ according to Lemma \ref{lemma-zusammenziehen}. We also recall the notations $H_{I,0}^+ := T(H_I^+)$ and $\calx_0 := H_{I,0}^+ \oplus H_J$. 

\begin{remark}\label{rem-I-finite}
If $I$ is finite, then we simply take $\nu_i := 1$ for all $i \in I$. In this case, the operator $T$ is the identity operator, and we have $H_{I,0} = H_I$, $H_{I,0}^+ = H_I^+$ and $\calx_0 = \calx$. 
\end{remark}

The following two results show that the drift and the volatility appearing in the SDE (\ref{SDE-affine-Y-I}) satisfy the linear growth condition with respect to the norm $\| \cdot \|_{H_{I,0}}$.

\begin{proposition}\label{prop-H0-mu}
We have $\bar{\mu}_{II}(H_{I,0}^+) \subset H_{I,0}$, and $\bar{\mu}_{II}|_{H_{I,0}^+} : H_{I,0}^+ \to H_{I,0}$ satisfies the linear growth condition with respect to $\| \cdot \|_{H_{I,0}}$.
\end{proposition}

\begin{proof}
By (\ref{ass-prin-drift}) and Lemma \ref{lemma-ass-3} we have $\bar{m}_{0,I} \in H_{I,0}^+$ and $\bar{M}_{II} T = T \bar{M}_{II}$. Hence, by Lemma \ref{lemma-function-commute} we have $\bar{M}_{II}(H_{I,0}) \subset H_{I,0}$, and $\bar{M}_{II}|_{H_{I,0}} \in L(H_{I,0})$ with respect to $\| \cdot \|_{H_{I,0}}$. Therefore, taking into account (\ref{def-mu-bar-I}) we have $\bar{\mu}_{II}(H_{I,0}^+) \subset H_{I,0}$ and the linear growth condition.
\end{proof}

Recalling the representation (\ref{T-compact-repr}), the system $(f_i)_{i \in I}$ given by
\begin{align*}
f_i = T e_i = \nu_i e_i, \quad i \in I
\end{align*}
is an orthonormal basis of $H_{I,0}$. Also recall that system $( g_k )_{k \in \bbn}$ given by (\ref{ONB-gk}) is an orthonormal basis of $U_0$. In view of the upcoming result, we emphasize that the spaces $H_{I,0}$ and $U_{I,0}$ have to be distinguished, although we have $H = U$. Indeed, by definition we have $H_{I,0} = T(H_I)$, where $T$ is given by (\ref{T-compact-repr}), and we have $U_{I,0} = \Sigma_W^{\nicefrac{1}{2}}(H_I)$.

\begin{proposition}\label{prop-H0-sigma}
The following statements are true:
\begin{enumerate}[(i)]
\item For all $y \in H_{I,0}^+$ and all $u \in U_{I,0}$ we have $\bar{\sigma}_{II}(y)u \in H_{I,0}$.

\item For all $y \in H_{I,0}^+$ we have $\bar{\sigma}_{II}(y) \in L_2(U_{I,0},H_{I,0})$ with representation
\begin{align*}
\bar{\sigma}_{II}(y)u = \sum_{i \in I} \sqrt{\frac{\lambda_i}{\nu_i}} \sqrt{\la y,f_i \ra_{H_0}} \, \la g_i,u \ra_{U_0} \, f_i, \quad u \in U_{I,0}.
\end{align*}

\item The mapping $\bar{\sigma}_{II}|_{H_{I,0}^+} : H_{I,0}^+ \to L_2(U_{I,0},H_{I,0})$ satisfies the linear growth condition with respect to $\| \cdot \|_{H_{I,0}}$.
\end{enumerate}
\end{proposition}

\begin{proof}
Let $y \in H_{I,0}^+$ and $u \in U_{I,0}$ be arbitrary. By Proposition \ref{prop-sigma} and Lemma \ref{lemma-zusammenziehen} we have
\begin{align*}
\bar{\sigma}_{II}(y)u &= \sum_{i \in I} \sqrt{\lambda_i \la y,e_i \ra_H } \, \la g_i, u \ra_{U_0} \, e_i
\\ &= \sum_{i \in I} \sqrt{\lambda_i \nu_i \la y,f_i \ra_{H_{I,0}} } \, \la g_i, u \ra_{U_0} \, e_i = \sum_{i \in I} \sqrt{\frac{\lambda_i}{\nu_i}} \sqrt{\la y,f_i \ra_{H_{I,0}}} \la g_i,u \ra_{U_0} \, f_i.
\end{align*}
Therefore, since $( \lambda_i / \nu_i )_{i \in I} \in \ell^2(I)$ we obtain
\begin{align*}
\sum_{i \in I} \frac{1}{\nu_i^2} | \la \bar{\sigma}_{II}(y)u, e_i \ra_{H} |^2 = \sum_{i \in I} \frac{\lambda_i}{\nu_i} \la y,f_i \ra_{H_{I,0}} | \la g_i,u \ra_{U_0} |^2 \leq \bigg\| \bigg( \frac{\lambda_i}{\nu_i} \bigg)_{i \in I} \bigg\|_{\ell^2(I)} \| y \|_{H_{I,0}} \| u \|_{U_0}^2 < \infty,
\end{align*}
and by Lemma \ref{lemma-zusammenziehen} it follows that $\bar{\sigma}_{II}(y)u \in H_{I,0}$. Recall that $(g_i)_{i \in I}$ is an orthonormal basis of $U_{I,0}$. Hence, by Lemma \ref{lemma-zusammenziehen} we obtain
\begin{align*}
\| \bar{\sigma}_{II}(y) \|_{L_2(U_{I,0},H_{I,0})}^2 = \sum_{i \in I} \| \bar{\sigma}_{II}(y) g_i \|_{H_{I,0}}^2 = \sum_{i \in I} \frac{\lambda_i}{\nu_i} \la y,f_i \ra_{H_{I,0}} \leq \bigg\| \bigg( \frac{\lambda_i}{\nu_i} \bigg)_{i \in I} \bigg\|_{\ell^2(I)} \| y \|_{H_{I,0}},
\end{align*}
proving that $\bar{\sigma}_{II}(y) \in L_2(U_{I,0},H_{I,0})$. Moreover, it follows that $\bar{\sigma}_{II}|_{H_{I,0}^+}$ satisfies the linear growth condition with respect to $\| \cdot \|_{H_{I,0}}$.
\end{proof}

The following two results provide the existence of weak solutions and pathwise uniqueness of solutions for the affine SDE (\ref{SDE-affine}), which we require in order to apply the Yamada-Watanabe theorem. Note that the drift $\bar{\mu}_J$ in the SDE (\ref{SDE-affine-Y-J}) is Lipschitz continuous, which easily follows by taking into account (\ref{def-mu-bar-J}).

\begin{proposition}\label{prop-affine-existence}
Suppose that condition (\ref{ass-prin-drift}) is fulfilled. Then for each probability measure $\nu$ on $(\calx_0,\ccB(\calx_0))$ there exists a weak solution to the affine SDE (\ref{SDE-affine}) such that $\nu = \bbp \circ X_0$.
\end{proposition}

\begin{proof}
Let $\nu$ be a probability measure on $(\calx_0,\ccB(\calx_0))$. Using Lemma \ref{lemma-Lambda-properties}, we have $\Lambda(\calx_0) = \calx_0$. Hence, setting $\bar{\nu} := \nu \circ \Lambda$, that is $\bar{\nu}(B) = \nu(\Lambda^{-1}B)$ for all $B \in \ccB(\calx_0)$, defines another probability measure on $(\calx_0,\ccB(\calx_0))$. Let $\bar{\nu}_I$ be the probability measure on $(H_{I,0}^+,\ccB(H_{I,0}^+))$ given by $\bar{\nu}_I(B) = \bar{\nu}(B \times H_{J})$ for all $B \in \ccB(H_{I,0})$. In view of Lemma \ref{lemma-lin-growth-sigma} and Propositions \ref{prop-H0-mu}, \ref{prop-H0-sigma} we may apply Theorem \ref{thm-SDE-existence}, which provides a weak solution $(Y_I,W)$ to the affine SDE (\ref{SDE-affine-Y-I}) on some stochastic basis $\bbb$ such that $\bar{\nu}_I = \bbp \circ Y_{I,0}$. By desintegration, there exists a stochastic kernel $K$ from $(H_{I,0}^+,\ccB(H_{I,0}^+))$ into $(H_{J},\ccB(H_{J}))$ such that $\bar{\nu} = K \otimes \bar{\nu}_I$. We define the probability kernel $\bar{K}$ from $(\Omega,\calf_0)$ into $(H_{J},\ccB(H_{J}))$ as
\begin{align*}
\bar{K}(\omega,B) := K(Y_{I,0}(\omega),B) \quad \text{for all $\omega \in \Omega$ and $B \in \ccB(H_{J})$.}
\end{align*}
Furthermore, we define the new stochastic basis
\begin{align*}
\widetilde{\bbb} := \big( \widetilde{\Omega}, \widetilde{\calf}, (\widetilde{\calf}_t)_{t \in \bbr_+}, \widetilde{\bbp} \big) := \big( \Omega \times H_{J}, \calf \otimes \ccB(H_{J}), (\calf_t \otimes \ccB(H_{J}))_{t \in \bbr_+} , \bar{K} \otimes \bbp \big),
\end{align*}
and we define the $\widetilde{\calf}_0$-measurable random variable $Y_0 : \widetilde{\Omega} \to \calx_0$ as
\begin{align*}
Y_0(\omega) := \big( Y_{I,0}(\omega_1),Y_{J,0}(\omega_2) \big), \quad \omega \in \widetilde{\Omega},
\end{align*}
where $Y_{J,0}(\omega_2) := \omega_2$. Then it is easy to verify that $\widetilde{\bbp} \circ Y_0 = \bar{\nu}$. There exists a strong solution $Y_J$ to the affine SDE (\ref{SDE-affine-Y-J}) on the stochastic basis $\widetilde{\bbb}$ with initial condition $Y_{J,0}$. Setting $Y := (Y_I,Y_J)$, this gives us a weak solution $(Y,W)$ to the affine SDE (\ref{SDE-affine-Y}) such that $\widetilde{\bbp} \circ Y_0 = \bar{\nu}$. Now, setting $X := \Lambda^{-1} Y$ we deduce that $(X,W)$ is a weak solution to the affine SDE (\ref{SDE-affine}). Moreover, we have $\widetilde{\bbp} \circ X_0 = \nu$, completing the proof.
\end{proof}

\begin{proposition}\label{prop-affine-unique}
Suppose that condition (\ref{trace-cond-drift}) is fulfilled. Then we have pathwise uniqueness with starting points in $\calx$ for the affine SDE (\ref{SDE-affine}).
\end{proposition}

\begin{proof}
Note that $\bar{\mu}_{II}$ has the affine structure (\ref{def-mu-bar-I}), and that $\bar{\sigma}_{II}$ has the diagonal structure (\ref{repr-sigma-block}) from Proposition \ref{prop-sigma}. Also noting (\ref{trace-cond-drift}), we may apply Theorem \ref{thm-SDE-uniqueness}, which provides pathwise uniqueness with starting points in $H_I^+$ for the affine SDE (\ref{SDE-affine-Y-I}). Moreover, for every weak solution $Y_I$ to (\ref{SDE-affine-Y-I}) we have pathwise uniqueness with starting points in $H_J$ for the affine SDE (\ref{SDE-affine-Y-J}). Consequently, we have pathwise uniqueness with starting points in $\calx$ for the affine SDE (\ref{SDE-affine-Y}). Since $\Lambda \in L(H)$ is an isomorphism, we deduce that pathwise uniqueness with starting points in $\calx$ for the affine SDE (\ref{SDE-affine}) holds.
\end{proof}

\begin{remark}
Note that we cannot apply Thm. 2.1 from \cite{Xie} in order to derive pathwise uniqueness. Indeed, defining $\psi_1,\varphi,\psi : \bbr_+ \to \bbr_+$ as $\psi_1(\theta) := \psi(\theta) := \theta$ and $\varphi(\theta) := \sqrt{\theta}$ for $\theta \in \bbr_+$, the integral divergence condition from (A4) in \cite{Xie} is not fulfilled, because
\begin{align*}
\int_0^1 \frac{1}{\psi(\theta) + \varphi(\theta)} d\theta &= \int_0^1 \frac{1}{\theta + \theta^{\nicefrac{1}{2}}} d\theta = \int_0^1 \frac{1}{\theta^{\nicefrac{1}{2}} ( \theta^{\nicefrac{1}{2}} + 1 )} d\theta
= \int_0^1 \bigg( \frac{1}{\theta^{\nicefrac{1}{2}}} - \frac{1}{\theta^{\nicefrac{1}{2}} + 1} \bigg) d\theta < \infty.
\end{align*}
\end{remark}

Noting that $\ccB(\calx_0) = \ccB(\calx)_{\calx_0}$ by (\ref{Borel-fields-Hilbert}), the proof of Theorem \ref{thm-main-ex} is now an immediate consequence of Propositions \ref{prop-affine-existence} and \ref{prop-affine-unique}, combined with our version of the Yamada-Watanabe theorem (Theorem \ref{thm-YW}).

\begin{remark}
In this paper we have focused on affine diffusion processes on the canonical state space. We suspect that our findings can also be transferred to affine processes with L\'{e}vy noise, and we keep this for future research. Apart from Ornstein-Uhlenbeck processes (see Section \ref{sec-ex-2}), the existence of such processes requires further investigations. In particular, it needs to be checked whether the available existence and uniqueness results can be applied to the additional jump part.
\end{remark}

\section{Examples}\label{sec-examples}

In this section, we present examples where Theorem \ref{thm-main-ex} applies. In Subsections \ref{sec-ex-1} and \ref{sec-ex-2} the application is straightforward. More care is required in Subsections \ref{sec-ex-3} and \ref{sec-ex-4}, where we consider infinite dimensional processes of Cox-Ingersoll-Ross type and of Heston type. In both of these two examples, we will first specify the volatility structure, then we define the retracted subspace with compact embedding, and in the last step we introduce the drift.

\subsection{Particular situations}\label{sec-ex-1}

As a consequence of the following result, Theorem \ref{thm-main-ex} in particular applies when the subspace $H_I$ is finite dimensional.

\begin{corollary}\label{cor-main-ex}
If the index set $I$ is finite, then the affine SDE (\ref{SDE-affine}) has a unique strong solution starting in $\calx$.
\end{corollary}

\begin{proof}
Since $I$ is finite, condition (\ref{frac-in-l2}) is fulfilled. Furthermore, by Remark \ref{rem-I-finite} the compact linear operator $T \in K(H_I)$ given by (\ref{T-compact-repr}) is the identity operator, and we have $H_{I,0} = H_I$, $H_{I,0}^+ = H_I^+$ and $\calx_0 = \calx$. Therefore, Assumption \ref{ass-ex} is satisfied. Consequently, applying Theorem \ref{thm-main-ex} completes the proof.
\end{proof}

In particular, Corollary \ref{cor-main-ex} applies when the Hilbert space $H$ is finite dimensional. Therefore, we have generalized \cite[Thm. 8.1]{Filipovic-Mayerhofer}, which provides the existence of affine processes in finite dimension.

\subsection{Infinite dimensional processes of Ornstein-Uhlenbeck type}\label{sec-ex-2}

Note that Corollary \ref{cor-main-ex} in particular applies when $I = \emptyset$, which provides the existence of affine processes with state space being the whole Hilbert space $H$. In this situation, where the process $X$ is a so-called Ornstein-Uhlenbeck process, we can say even more. By condition (\ref{cond-parallel-2}) from Proposition \ref{prop-parallel} the mapping $S$ is constant, and hence the volatility $\sigma$ given by (\ref{sigma-S}) is constant as well. Therefore, the affine SDE (\ref{SDE-affine}) has the explicit solution given by the variation of constants formula
\begin{align*}
X_t = x_0 + \int_0^t S_{t-s} m_0 ds + \int_0^t S_{t-s} \sigma dW_s, \quad t \in \bbr_+,
\end{align*}
where $(S_t)_{t \geq 0}$ denotes the uniformly continuous semigroup generated by the linear operator $M \in L(H)$ appearing in (\ref{mu-affine}). Ornstein-Uhlenbeck processes on Hilbert spaces have recently been studied in \cite{BenthRuedigerSuess2018} and \cite{BenthSimonsen2018}, and they provide a link to the theory of semilinear stochastic partial differential equations (SPDEs) in the spirit of the semigroup approach; see for example \cite{DaPratoZabczyk}. More precisely, in the general situation, where $I$ and $J$ are arbitrary disjoint index sets, we could also regard the affine SDE (\ref{SDE-affine}) as a SPDE and look for (mild) martingale solutions, which means that the variation of constants formula
\begin{align*}
X_t = S_t x_0 + \int_0^t S_{t-s} m_0 ds + \int_0^t S_{t-s} \sigma(X_s) dW_s, \quad t \in \bbr_+
\end{align*}
is satisfied. According to \cite[Thm. 8.1]{DaPratoZabczyk}, a sufficient condition for the existence of martingale solutions is that the semigroup $(S_t)_{t \geq 0}$ is compact. However, we have the following negative result.

\begin{proposition}
Suppose that $\dim H = \infty$. Then no uniformly continuous, compact semigroup exists.
\end{proposition}

\begin{proof}
Suppose that $A \in L(H)$ is the infinitesimal generator of a compact semigroup. According to \cite[Prop. 3.1.4]{Pazy92}, for every $B \in L(H)$ the operator $A+B$ is also the generator of a compact semigroup. In particular, choosing $B = -A$ we obtain that the semigroup $(S_t)_{t \geq 0}$ given by $S_t = \Id$ for each $t \geq 0$ is compact. Since $\dim H = \infty$, this is a contradiction.
\end{proof}

Consequently, apart from the particular case of Ornstein-Uhlenbeck processes, the SPDE approach is not appropriate in order to establish the existence of affine processes, and this is why we regard affine processes as solutions to infinite dimensional SDEs in this paper.

\subsection{Infinite dimensional processes of Cox-Ingersoll-Ross type}\label{sec-ex-3}

In this subsection, we establish the existence of infinite dimensional processes of Cox-Ingersoll-Ross type. Here we have the index sets $I = \bbn$ and $J = \emptyset$, and hence the state space is given by $\calx = H^+$. We define the volatility structure $S : H^+ \to L_1^+(H)$, the retracted subspace $H_0$ and the drift $\mu : H^+ \to H$ in three steps:

\subsubsection*{Volatility structure} Let $\lambda = (\lambda_i)_{i \in \bbn} \subset (0,\infty)$ be a sequence such that $\lambda \in \ell^2(\bbn)$. Then the mapping
\begin{align*}
N(x)\xi := \sum_{i \in \bbn} \lambda_i \la x,e_i \ra_H \la e_i,\xi \ra_H \, e_i, \quad x,\xi \in H
\end{align*}
is a well-defined continuous linear operator $N \in L(H,L_1(H))$. Furthermore, for every $x \in H^+$ the operator $N(x)$ is self-adjoint with $N(x) \in L_1^+(H)$. We define the affine mapping $S : H^+ \to L_1^+(H)$ as $S(x) := N(x)$ for all $x \in H^+$. Then $S$ has the affine form (\ref{S-affine}) with $n_0 = 0$, and we have
\begin{align*}
\lambda_i = \| S(e_i)e_i \|_H \quad \text{for each $i \in \bbn$.}
\end{align*}
Note that conditions (\ref{cond-parallel-1})--(\ref{cond-parallel-3}) from Proposition \ref{prop-parallel} are fulfilled, ensuring that the associated volatility $\sigma : H^+ \to L^2(U_0,H)$ given by (\ref{sigma-S}) is parallel. Furthermore, condition (\ref{frac-in-l2}) is fulfilled, because $\kappa_i = 0$ for each $i \in \bbn$, and the transformation $\Lambda \in L(H)$ given by (\ref{Lambda-def}) is simply the identity operator. By Proposition \ref{prop-sigma} the volatility is given by
\begin{align*}
\sigma(x) u = \sum_{i \in \bbn} \sqrt{ \lambda_i \la x,e_i \ra_H } \, \la g_i,u \ra_{U_0} \, e_i \quad \text{for all $x \in H^+$ and $u \in U_0$,}
\end{align*}
where $g_i = \Sigma_W^{\nicefrac{1}{2}} e_i$ for all $i \in \bbn$. If we even have $\lambda \in \ell^1(\bbn)$, then we can take the covariance operator $\Sigma_W \in L_1^{++}(U)$ defined as
\begin{align*}
\Sigma_W \, u := \sum_{i \in \bbn} \lambda_i \la e_i,u \ra_U \, e_i \quad \text{for all $u \in U$,}
\end{align*}
and then the volatility admits the representation
\begin{align*}
\sigma(x) u = \sum_{i \in \bbn} \sqrt{ \la x,e_i \ra_H } \, \la e_i,u \ra_U \, e_i \quad \text{for all $x \in H^+$ and $u \in U_0$.}
\end{align*}

\subsubsection*{Retracted subspace} Since $\lambda \in \ell^2(\bbn)$, there exists a sequence $(\nu_i)_{i \in \bbn} \subset (0,\infty)$ such that $\nu_i \to 0$ and $( \lambda_i / \nu_i )_{i \in \bbn} \in \ell^2(\bbn)$. Let $T \in K^{++}(H)$ be the compact linear operator with representation
\begin{align}\label{T-CIR}
Tx = \sum_{i \in \bbn} \nu_i \la x,e_i \ra_H \, e_i, \quad x \in H,
\end{align}
and let $H_0 := T(H)$ be the retracted subspace defined according to Lemma \ref{lemma-zusammenziehen}. Moreover, we set $H_0^+ := T(H^+)$.

\subsubsection*{Drift} We define the affine mapping $\mu : H^+ \to H$ as $\mu(x) := m_0 + Mx$, $x \in H^+$, where $m_0 \in H_0^+$, and $M \in L(H)$ is of the form
\begin{align}\label{M-CIR}
Mx = \sum_{i \in \bbn} \rho_i \la x,e_i \ra_H \, e_i, \quad x \in H
\end{align}
with a sequence $\rho = (\rho_i)_{i \in \bbn} \subset \bbr$ such that $\rho \in \ell^1(\bbn)$. Note that condition (\ref{trace-cond-drift}) and conditions (\ref{cond-inward-1})--(\ref{cond-inward-3}) from Proposition \ref{prop-inward} are fulfilled, which shows that $\mu$ is inward pointing. By the representations (\ref{T-CIR}) and (\ref{M-CIR}) we have $MT = TM$, showing that condition (\ref{ass-prin-drift}) is fulfilled. Consequently, by Theorem \ref{thm-main-ex} the affine Cox-Ingersoll-Ross type SDE (\ref{SDE-affine}) with parameters specified above has a unique strong solution starting in $H_0^+$.

\subsection{Infinite dimensional processes of Heston type}\label{sec-ex-4}

In this subsection, we establish the existence of infinite dimensional processes of Heston type. We assume that the disjoint index sets $I$ and $J$ are both infinite. Let $\tau : I \to J$ be a bijection which is order preserving; that is $\tau(i_1) \leq \tau(i_2)$ for all $i_1,i_2 \in I$ with $i_1 \leq i_2$. We define the volatility structure $S : \calx \to L_1^+(H)$, the retracted subspace $H_{I,0}$ and the drift $\mu : \calx \to H$ in three steps:

\subsubsection*{Volatility structure} Let $n_0 \in L_1^+(H)$ be such that $n_0 \xi = 0$ for all $\xi \in H_I$. Furthermore, let $\lambda = (\lambda_i)_{i \in I} \subset (0,\infty)$ and $\kappa = (\kappa_i)_{i \in I} \subset (0,\infty)$ be sequences such that $\lambda \in \ell^2(I)$ and $\kappa_i \leq \lambda_i$ for all $i \in I$. Moreover, we assume that $( \kappa_i / \lambda_i )_{i \in I} \in \ell^2(I)$. Then the mapping
\begin{align}\label{repr-N-Heston}
N(x)\xi := \sum_{i \in I} \la x,e_i \ra_H \Big( \la \lambda_i e_i + \kappa_i e_{\tau(i)},\xi \ra_H \, e_i + \la \kappa_i e_i + \lambda_i e_{\tau(i)}, \xi \ra_H \, e_{\tau(i)} \Big), \quad x,\xi \in H
\end{align}
is a well-defined continuous linear operator $N \in L(H,L_1(H))$. Furthermore, for every $x \in \calx$ the operator $Nx$ is self-adjoint with $Nx \in L_1^+(H)$. Noting that $\kappa_i \leq \lambda_i$ for all $i \in I$, this follows from the representation (\ref{repr-N-Heston}). We define the affine mapping $S : \calx \to L_1^+(H)$ as
\begin{align*}
S(x) := n_0 + Nx, \quad x \in \calx.
\end{align*}
Noting that $H_I \subset \ker(n_0)$, by the representation (\ref{repr-N-Heston}) we have
\begin{align*}
\lambda_i = \| S(e_i)_{II} \, e_i \|_H, \quad \text{and} \quad \kappa_i = \| S(e_i)_{IJ} \, e_i \|_H \quad \text{for each $i \in I$.}
\end{align*}
Note that conditions (\ref{cond-parallel-1})--(\ref{cond-parallel-3}) from Proposition \ref{prop-parallel} are fulfilled, ensuring that the associated volatility $\sigma : \calx \to L^2(U_0,H)$ given by (\ref{sigma-S}) is parallel. Furthermore, condition (\ref{frac-in-l2}) is fulfilled, and the transformation $\Lambda \in L(H)$ specified by (\ref{Lambda-def}) is given by $\Lambda = \Id + D$, where $D \in L(H)$ denotes the linear operator
\begin{align*}
Dx = - \sum_{i \in I} \frac{\kappa_i}{\lambda_i} \la x,e_i \ra_H \, e_{\tau(i)}, \quad x \in H.
\end{align*}

\subsubsection*{Retracted subspace} Since $\lambda \in \ell^2(I)$, there exists a sequence $(\nu_i)_{i \in I} \subset (0,\infty)$ such that $\nu_i \to 0$ and $( \lambda_i / \nu_i )_{i \in I} \in \ell^2(I)$. Let $T \in K(H_I)$ be the compact linear operator with representation (\ref{T-compact-repr}), and let $H_{I,0} := T(H_I)$ be the retracted subspace defined according to Lemma \ref{lemma-zusammenziehen}. Moreover, we set $H_{I,0}^+ := T(H_I^+)$ and $\calx_0 := H_{I,0}^+ \oplus H_J$.

\subsubsection*{Drift} Let $m_0 \in \calx$ be such that $m_{0,I} \in H_{I,0}^+$. Furthermore, let $M_{II} \in L(H_I)$ be a linear operator of the form
\begin{align}\label{M-Heston}
M_{II} x = \sum_{i \in I} \rho_i \la x,e_i \ra_H \, e_i, \quad x \in H_I
\end{align}
with a bounded sequence $\rho = (\rho_i)_{i \in I} \subset \bbr$ such that $\rho \in \ell^1(I)$, and let $M_J \in L(H,H_J)$ be arbitrary. We define $M \in L(H)$ as
\begin{align*}
Mx := M_{II} x_I + M_J x, \quad x \in H,
\end{align*}
and the affine mapping $\mu : \calx \to H$ by (\ref{mu-affine}). Note that condition (\ref{trace-cond-drift}) and conditions (\ref{cond-inward-1})--(\ref{cond-inward-3}) from Proposition \ref{prop-inward} are fulfilled, which shows that $\mu$ is inward pointing. By the representations (\ref{T-compact-repr}) and (\ref{M-Heston}) we have $M_{II}T = TM_{II}$, showing that condition (\ref{ass-prin-drift}) is fulfilled. Consequently, by Theorem \ref{thm-main-ex} the affine Heston type SDE (\ref{SDE-affine}) has a unique strong solution starting in $\calx_0$.

\begin{appendix}

\section{Proof of Lemma \ref{lemma 2.8}}\label{sec-aux}
We begin the proof of Lemma \ref{lemma 2.8} with the following remark: an 
	$A\in L(H; L(H))$ is considered to be a sesquilinear map from $H_\C \times H_\C$ into $H_\C$ by defining $\eta^* A\xi = \sum_{i=1}^\infty \la A_i \xi, \overline\eta\ra_{H_\C}e_i$ for all $\xi, \eta \in H_\C$. It is well-defined, if $(A_i)_{i\in\N}$ are trace class operators such that $\sum_{i\in I} (\tr A_i)^2 < \infty$ and $A_j =0 $ for all $j\in J$. Actually, we have 
	\begin{align*}
		\|\eta^* A\xi\|_{H_\C}^2 &= \sum_{i=1}^\infty|\la A_i \xi, \overline\eta\ra_{H_\C}|^2 \le \sum_{i=1}^\infty (\tr A_i)^2 \|\xi\|_{H_\C}^2\|\eta\|_{H_\C}^2 \\
		&= \sum_{i\in I} (\tr A_i)^2 \|\xi\|_{H_\C}^2\|\eta\|_{H_\C}^2 < \infty.
	\end{align*}
	Note that, for all $i \in \N$, $\la \psi(t,u), M_i\ra_{H_\C} = \la M^*\psi(t,u), e_i \ra_{H_\C}$ and $\la n_i\psi(t,u), \overline{\psi(t,u)} \ra_{H_\C} = \la \psi(t,u)^* N \psi(t,u), e_i \ra_{H_\C}$. Here, we use $M$ and $N$ from Equation \eqref{affcoeffe}.
	
	Since $\psi(t,u)$ is Fr\'echet differentiable and  $\la D_t\psi(t,u), e_i \ra_{H_\C} = \partial_t \psi_i(t,u)$, Equation  \eqref{Riccati2} is equivalent to  
	\begin{equation}
		\begin{split}
			D_t \psi(t,u) & = M^*\psi(t,u) + \frac12 \psi(t,u)^* N \psi(t,u),\quad t \ge 0, \\
			\psi(0,u) & = u \in \cU.
		\end{split}
		\label{equ:ats:riccati_equation_system_psi_var}
	\end{equation}

\begin{proof}[Proof of Lemma \ref{lemma 2.8}]
	We truncate $I$ at $n$ by considering $I_n = I \cap \{1,\ldots,n\},\ n\in\N$. Denote   $J_n = \N \setminus I_n$. By the very definition of $\la \cdot , \cdot \ra_{H_\bbC}$ we directly obtain that
	\begin{align}
		\nonumber\partial_t\|\psi_{I_n}(t,u)\|_{H_\C}^2 &= \sum_{i\in I_n}\partial_t\Big(\re\psi_i(t,u)^2 + \im\psi_i(t,u)^2\Big) \\
		\nonumber& =2\sum_{i\in I_n} \big(\re \psi_i(t,u)\partial_t(\re \psi_i(t,u)) + \im \psi_i(t,u)\partial_t(\im \psi_i(t,u)) \big) \\
		& =2\sum_{i\in I_n}\re \big(\overline{\psi_i(t,u)}\partial_t\psi_i(t,u)\big).
		\label{equ:ats:norm_psi_equation}
	\end{align} 
	It follows that for a bounded linear operator $L\in L(H)$ and $\psi \in H$, $\la L \psi, \overline{\psi} \ra_H + \la L \overline{\psi},\psi \ra_H = 2 \re \la L \psi, \overline{\psi} \ra_H$. Then, \eqref{equ:ats:riccati_equation_system_psi_var} together with \eqref{prop1:7} yields 
	\begin{align*}
	\partial_t\psi_i(t,u) &= \la \psi_{I_n}(t,u) + \psi_{J_n}(t,u), m_i\ra_{H_\C} \\ 
	&\hspace{.5cm} + \frac12 \la n_i (\psi_{I_n}(t,u) + \psi_{J_n}(t,u)), \psi_{I_n}(t,u) + \psi_{J_n}(t,u) \ra_{H_\C} \\
	&= \la \psi_{I_n}(t,u),m_{i,I_n}\ra_{H_\C} + \la \psi_{J_n}(t,u),m_{i,J_n}\ra_{H_\C} + \frac12 n_{i,ii}|\psi_i(t,u)|^2 \\ 
	& \hspace{.5cm} + \re\la n_{i,I_nJ_n}\psi_{J_n}(t,u),\overline{\psi_{I_n}(t,u)}\ra_{H_\C} 
	    + \frac12\la n_{i,J_nJ_n}\psi_{J_n}(t,u),\overline{\psi_{J_n}(t,u)}\ra_{H_\C},
	\end{align*}
	for all $i\in I_n$. Hence, for $t\in[0,T_u)$ and $i\in I_n$,
	\begin{align}
		\nonumber 2\re \big( \overline{\psi_i(t,u)}\partial_t\psi_i(t,u) \big) = & 2\re\Big(\overline{\psi_i(t,u)}\big(\la \psi_{I_n}(t,u),m_{i,I_n}\ra_{H_\C} + \la \psi_{J_n}(t,u),m_{i,J_n}\ra_{H_\C}\big)\Big) \\
		\nonumber & + \re\psi_i(t,u) n_{i,ii}|\psi_i(t,u)|^2 \\
		\nonumber & + 2\re \psi_i(t,u) \re\la n_{i,I_nJ_n}\psi_{J_n}(t,u),\overline{\psi_{I_n}(t,u)}\ra_{H_\C} \\
		& + \re\Big(\overline{\psi_i(t,u)}\la n_{i,J_nJ_n}\psi_{J_n}(t,u),\overline{\psi_{J_n}(t,u)}\ra_{H_\C}\Big). \label{equ:ats:partial_derivative_i}
	\end{align}
	Since $\re \psi_i(t,u) \le 0$ for $i\in I_n \subset I$ and $t\in[0,T_u)$ and $n_{i,ii} \ge 0$, we obtain that $\re \psi_i(t,u)n_{i,ii}|\psi_i(t,u)|^2 \le 0$ and it follows that
	\begin{align*}
	\eqref{equ:ats:partial_derivative_i} & \le  2\re\Big(\overline{\psi_i(t,u)}\big(\la \psi_{I_n}(t,u),m_{i,I_n}\ra_{H_\C} + \la \psi_{J_n}(t,u),m_{i,J_n}\ra_{H_\C}\big)\Big) \\
	& + 2\re \psi_i(t,u) \re\la n_{i,I_nJ_n}\psi_{J_n}(t,u),\overline{\psi_{I_n}(t,u)}\ra_{H_\C} \\
	& + \re\Big(\overline{\psi_i(t,u)}\la n_{i,J_nJ_n}\psi_{J_n}(t,u),\overline{\psi_{J_n}(t,u)}\ra_{H_\C}\Big).
	\intertext{Using that  $2\re (\alpha\beta) \le |\alpha|^2 + |\beta|^2$ and $2\re \alpha \re \beta \le |\alpha|^2 + |\beta|^2$ for $\alpha,\beta \in \C$, we get}
	\eqref{equ:ats:partial_derivative_i} & \le  |\psi_i(t,u)|^2 + |\la \psi_{I_n}(t,u),m_{i,I_n}\ra_{H_\C}|^2 + |\psi_i(t,u)|^2 + |\la \psi_{J_n}(t,u),m_{i,J_n}\ra_{H_\C}|^2 \\
	& + |\psi_i(t,u)|^2 + |\la n_{i,I_nJ_n}\psi_{J_n}(t,u),\overline{\psi_{I_n}(t,u)}\ra_{H_\C}|^2 \\ 
	& + \frac12|\psi_i(t,u)|^2 + \frac12|\la n_{i,J_nJ_n}\psi_{J_n}(t,u),\overline{\psi_{J_n}(t,u)}\ra_{H_\C}|^2.
	\intertext{Note that for $i \in I_n$, $\la \psi_{I_n}(t,u),m_{i,I_n}\ra_{H_\C} = \la \psi_{I_n}(t,u),\pi_{I_n} M e_i\ra_{H_\C} = \la M^*\psi_{I_n}(t,u), e_i\ra_{H_\C}$ and that this is also true when $I_n$ is replaced by $J_n$. Hence,}
	\eqref{equ:ats:partial_derivative_i} & \le  \frac72|\psi_i(t,u)|^2 + |\la M^*\psi_{I_n}(t,u),e_i\ra_{H_\C}|^2 + |\la M^*\psi_{J_n}(t,u), e_i\ra_{H_\C}|^2 \\
	& + |\la n_{i,I_nJ_n}\psi_{J_n}(t,u),\overline{\psi_{I_n}(t,u)}\ra_{H_\C}|^2 
	  + \frac12|\la n_{i,J_nJ_n}\psi_{J_n}(t,u),\overline{\psi_{J_n}(t,u)}\ra_{H_\C}|^2.
	\end{align*}
	Moreover,  $\|n_{i,KL}\| \le \| n_i \|$ for any subset $K,L \subset \N$, such that
	\begin{align*}
	2\re \big( \overline{\psi_i(t,u)}\partial_t\psi_i(t,u) \big)
	& \le \frac72|\psi_i(t,u)|^2 + |\la M^*\psi_{I_n}(t,u),e_i\ra_{H_\C}|^2
	+ |\la M^*\psi_{J_n}(t,u),e_i\ra_{H_\C}|^2 \\
	&\hspace{.5cm} + \|n_i\|^2\big(\|\psi_{J_n}(t,u)\|_{H_\C}^2\|\psi_{I_n}(t,u)\|_{H_\C}^2+ \frac12\|\psi_{J_n}(t,u)\|_{H_\C}^4\big).
	\end{align*}
	It follows from \eqref{equ:ats:norm_psi_equation} that
	\begin{align*}
	\partial_t\|\psi_{I_n}(t,u)\|_{H_\C}^2 &\le \sum_{i\in I_n} \left( \frac72|\psi_i(t,u)|^2 + |\la M^*\psi_{I_n}(t,u),e_i\ra_{H_\C}|^2
	+ |\la M^*\psi_{J_n}(t,u),e_i\ra_{H_\C}|^2 \right) \\
	&\hspace{.5cm}+ \sum_{i\in I_n} \|n_i\|^2\big(\|\psi_{J_n}(t,u)\|_{H_\C}^2\|\psi_{I_n}(t,u)\|_{H_\C}^2 +  \|\psi_{J_n}(t,u)\|_{H_\C}^4\big) \\
	&\le \frac72 \|\psi_{I_n}(t,u)\|_{H_\C}^2 + \|M^*\psi_{I_n}(t,u)\|_{H_\C}^2 + \|M^*\psi_{J_n}(t,u)\|_{H_\C}^2 \\ 
	&\hspace{.5cm}+ \sum_{i\in I_n} \|n_i\|^2\big(\|\psi_{J_n}(t,u)\|_{H_\C}^2\|\psi_{I_n}(t,u)\|_{H_\C}^2 +  \|\psi_{J_n}(t,u)\|_{H_\C}^4\big) \\
	&\le \frac72 \|\psi_{I_n}(t,u)\|_{H_\C}^2 + \|M\|^2\|\psi_{I_n}(t,u)\|_{H_\C}^2 + \|M\|^2\|\psi_{J_n}(t,u)\|_{H_\C}^2 \\ 
	&\hspace{.5cm}+ \sum_{i\in I_n} \|n_i\|^2\big(\|\psi_{J_n}(t,u)\|_{H_\C}^2\|\psi_{I_n}(t,u)\|_{H_\C}^2 +  \|\psi_{J_n}(t,u)\|_{H_\C}^4\big) \\
	&\le C \big(\|\psi_{I_n}(t,u)\|_{H_\C}^2 + \|\psi_{J_n}(t,u)\|_{H_\C}^2 + \|\psi_{J_n}(t,u)\|_{H_\C}^2\|\psi_{I_n}(t,u)\|_{H_\C}^2 \\
	&\hspace{.5cm}+ \|\psi_{J_n}(t,u)\|_{H_\C}^4\big) \\
	&\le C \big(1+\|\psi_{I_n}(t,u)\|_{H_\C}^2\big)\big(1+\|\psi_{J_n}(t,u)\|_{H_\C}^2+\|\psi_{J_n}(t,u)\|_{H_\C}^4\big) 
	\end{align*}
	holds for all $t\in[0,T_u)$, where $C = \sum_{i\in I} \|n_i\|^2 + \|M\|^2 + \frac72$. 
    We set $h_{n,u}(t):= 1+\|\psi_{J_n}(t,u)\|_{H_\C}^2+\|\psi_{J_n}(t,u)\|_{H_\C}^4$
	Applying Gronwall's inequality
	yields that for $t\in[0,T_u)$,
	\begin{equation}
		1+\|\psi_{I_n}(t,u)\|_{H_\C}^2 \le 1+\|u_{I_n}\|_{H_\C}^2 + C(1+\|u_{I_n}\|_{H_\C}^2)\int_0^t h_{n,u}(s) e^{C\int_s^t h_{n,u}(r) d r} d s.
		\label{equ:ats:gronwall_n}
	\end{equation}
	Let $t\in[0,T_u)$. By the definitions of $I_n$ and $J_n$, it holds that	$\lim_{n\to\infty}\|\psi_{I_n}(t,u)\|_{H_\C}=\|\psi_I(t,u)\|_{H_\C}$, $\lim_{n\to\infty}u_{I_n}=u_I$ increasingly and $\lim_{n\to\infty}\|\psi_{J_n}(t,u)\|_{H_\C}=\|\psi_J(t,u)\|_{H_\C}$ decreasingly. 
	It follows that $\lim_{n\to\infty}h_{n,u}(r)=h_u(r)$ decreasingly for all $r \in [s,t]$. By Dini's theorem this convergence is uniform over $[s,t]$, thus, $\lim_{n\to\infty}\int_s^t h_{n,u}(r)d r=\int_s^t h_u(r)d r$ decreasingly for all $s \in [0,t]$. 
	
	Consider the sequence of functions $(h_{n,u}(s)\exp( \int_s^t C h_{n,u}(r)d r))_{n\in\N}$. Similar to the discussion above, this sequence converges to $h_u(s)\exp(C\int_s^t h_u(r)d r)$ decreasingly for all $s\in[0,t]$. Again, by Dini's Theorem, this convergence is uniform on $[0,t]$ and therefore it holds that 
	$$
	\lim_{n\to\infty}\int_0^t h_{n,u}(s) e^{C\int_s^t h_{n,u}(r)d r} d s = \int_0^t h_u(s) e^{\int_s^t h_u(r)d r} d s,\quad t \in [0,T_u).
	$$
	Taking limits on both sides of \eqref{equ:ats:gronwall_n}, we finally get
	$$
	1+\|\psi_I(t,u)\|_{H_\C}^2 \le 1+\|u_I\|_{H_\C}^2 + C(1+\|u_I\|_{H_\C}^2)\int_0^t h_u(s) e^{C \int_s^t h_u(r) d r} d s,\quad t\in[0,T_u).
	$$
	Subtracting $1$ from both sides, the required inequality is proved.
\end{proof}

\section{Infinite dimensional stochastic differential equations}\label{app-SDE}

The goal of this appendix is to provide the required results about the existence of solutions to infinite dimensional SDEs. In particular, we present a version of the Yamada-Watanabe theorem for starting points from a subspace which is equipped with a finer topology. This version of the Yamada-Watanabe theorem is inspired by \cite{GMR}, where an existence result for starting points from a retracted subspace with compact embedding is presented. We will also provide a refined version of this existence result as well as a result for pathwise uniqueness, which is a version of the uniqueness result from \cite{Yamada-Watanabe-1971} in infinite dimension.

Let $H$ be a separable Hilbert space, and let $\calx \subset H$ be a subset. Let $U$ be a separable Hilbert space and let $\Sigma_W \in L_1^{++}(U)$ be a self-adjoint, strictly positive trace class operator. By Lemma \ref{lemma-zusammenziehen} the set $U_0 := \Sigma_W^{\nicefrac{1}{2}} U$, equipped with the inner product
\begin{align*}
\la u,v \ra_{U_0} := \big\la \Sigma_W^{-\nicefrac{1}{2}} u, \Sigma_W^{-\nicefrac{1}{2}} v \big\ra_U, \quad u,v \in U_0
\end{align*}
is a separable Hilbert space. Let $\mu : \calx \to H$ and $\sigma : \calx \to L_2(U_0,H)$ be measurable mappings. We consider the SDE
\begin{align}\label{SDE-Hilbert}
\left\{
\begin{array}{rcl}
dX_t & = & \mu(X_t)dt + \sigma(X_t) dW_t
\\ X_0 & = & x_0,
\end{array}
\right.
\end{align}
where $W$ is a $U$-valued Wiener process with covariance operator $\Sigma_W$. Let $(H_0,\| \cdot \|_{H_0})$ be a separable Hilbert space as in Lemma \ref{lemma-zusammenziehen}. Then we have $H_0 \subset H$ as a set, and by Lemma \ref{lemma-zusammenziehen} we have $H_0 \in \ccB(H)$ and $\ccB(H_0) = \ccB(H)_{H_0}$. We define $\calx_0 := \calx \cap H_0$, and denote by $\ccB(\calx)$ the Borel $\sigma$-algebra of $\calx$ with respect to $\| \cdot \|_{H}$, and by $\ccB(\calx_0)$ the Borel $\sigma$-algebra of $\calx_0$ with respect to $\| \cdot \|_{H_0}$. Then we have 
\begin{align}\label{Borel-fields-Hilbert}
\ccB(\calx_0) = \ccB(H)_{\calx_0} = \ccB(\calx)_{\calx_0}, 
\end{align}
because $\ccB(\calx_0) = \ccB(H_0)_{\calx} = (\ccB(H)_{H_0})_{\calx} = \ccB(H)_{\calx_0} = \ccB(\calx)_{\calx_0}$. In the sequel, we will speak about weak solutions, pathwise uniqueness, strong solutions and a unique strong solution \emph{starting in $\calx_0$}. These are the usual definitions (see, for example \cite{Roeckner}), but here we consider starting points from $\calx_0$, which is equipped with the Borel $\sigma$-algebra $\ccB(\calx_0)$.

\begin{theorem}\label{thm-YW}
The SDE (\ref{SDE-Hilbert}) has a unique strong solution starting in $\calx_0$ if and only if both of the following two conditions are satisfied:
\begin{enumerate}[(i)]
\item For each probability measure $\nu$ on $(\calx_0,\ccB(\calx_0))$ there exists a weak solution to (\ref{SDE-Hilbert}) such that $\nu = \bbp \circ X_0$.

\item Pathwise uniqueness for solutions to (\ref{SDE-Hilbert}) starting in $\calx_0$ holds.
\end{enumerate}
\end{theorem}

\begin{proof}
We provide a sketch of the proof. The necessity of conditions (i) and (ii) is straightforward, and follows as in the beginning of the proof of \cite[Thm. 2.1]{Roeckner}. The sufficiency of conditions (i) and (ii) is proven in two steps:
\begin{enumerate}
\item[(a)] First we show that for each probability measure $\nu$ on $(\calx_0,\ccB(\calx_0))$ there exists a unique strong solution to (\ref{SDE-Hilbert}) such that $\nu = \bbp \circ X_0$. Noting (\ref{Borel-fields-Hilbert}), this follows from \cite[Lemma 2.10]{Roeckner}.

\item[(b)] Then we can proceed as in the proof of \cite[Lemma 2.11]{Roeckner} in order to obtain a unique strong solution to (\ref{SDE-Hilbert}) starting in $\calx_0$.
\end{enumerate}
\end{proof}

\begin{remark}
For related work about the Yamada-Watanabe theorem in infinite dimension we refer to \cite{Ondrejat}, \cite{Roeckner} (see also Appendix E in \cite{Liu-Roeckner}) and \cite{Tappe-YW} as well as \cite{Kurtz-2007} and the successive paper \cite{Kurtz-2014}. Theorem \ref{thm-YW} does not immediately follow from these papers, because here we are dealing with solutions having values in $\calx$, but starting in $\calx_0$, where the latter set is equipped with a finer topology. The reason for this adjusted version of the Yamada-Watanabe theorem is our upcoming existence result; see Theorem \ref{thm-SDE-existence}. However, as seen from the proof of Theorem \ref{thm-YW}, the existence of strong solutions with a given probability measure $\nu$ on $(\calx_0,\ccB(\calx_0))$ as initial distribution -- part (a) -- follows from the aforementioned references. For this result, we can, for example, also refer to \cite[Thm. 3.14]{Kurtz-2007} or \cite[Thm. 1.5]{Kurtz-2014}. Moreover, the construction of the unique strong solution -- part (b) -- involves well-known techniques.
\end{remark}

Now, we present sufficient conditions for an application of our version of the Yamada-Watanabe theorem (Theorem \ref{thm-YW}). We start with sufficient conditions for the existence of weak solutions. Here we present a refined version of a result from \cite{GMR}, where the essential idea is to consider starting points from a retracted subspace with compact embedding. For the rest of this section, we assume that the set $\calx$ is of the form
\begin{align*}
\calx = \lin^+ \{ e_k : k \in \bbn \}
\end{align*}
for some orthonormal basis $( e_k )_{k \in \bbn}$ of $H$.

\begin{theorem}\label{thm-SDE-existence}
We suppose that the following conditions are fulfilled:
\begin{enumerate}[(i)]
\item The compact linear operator $T \in K(H)$ from Lemma \ref{lemma-zusammenziehen} has the representation (\ref{T-repr}) with respect to the given orthonormal basis $( e_k )_{k \in \bbn}$.

\item The mappings $\mu : \calx \to H$ and $\sigma : \calx \to L_2(U_0,H)$ are continuous and satisfy the linear growth condition.

\item We have $\mu(\calx_0) \subset H_0$ and $\sigma(\calx_0) \subset L_2(U_0,H_0)$, and the mappings $\mu|_{\calx_0} : \calx_0 \to H_0$ and $\sigma|_{\calx_0} : \calx_0 \to L_2(U_0,H_0)$ satisfy the linear growth condition with respect to $\| \cdot \|_{H_0}$.

\item The mapping $\mu$ is inward pointing at boundary points of $\calx$, and the mapping $\sigma$ is parallel to the boundary at boundary points of $\calx$.
\end{enumerate}
Then for each probability measure $\nu$ on $(\calx_0,\ccB(\calx_0))$ there exists a weak solution $(X,W)$ to the SDE (\ref{SDE-Hilbert}) such that $\nu = \bbp \circ X_0$.
\end{theorem}

\begin{proof}
Let $\Pi : H \to \calx$ be the metric projection on the closed convex cone $\calx$. Then we have
\begin{align*}
\Pi x = \sum_{i \in \bbn} \la x,e_i \ra_H^+ \, e_i \quad \text{for each $x \in H$,}
\end{align*}
and $\Pi$ is continuous and satisfies $\| \Pi x \|_H \leq \| x \|_H$ for all $x \in H$. Furthermore, we have $\Pi T = T \Pi$. By Lemma \ref{lemma-function-commute} we have $\Pi(H_0) \subset H_0$, and $\Pi|_{H_0} : H_0 \to H_0$ satisfies the linear growth condition with respect to $\| \cdot \|_{H_0}$. Consider the $H$-valued SDE
\begin{align}\label{SDE-Hilbert-full}
\left\{
\begin{array}{rcl}
dX_t & = & \bar{\mu}(X_t)dt + \bar{\sigma}(X_t) dW_t
\\ X_0 & = & x_0,
\end{array}
\right.
\end{align}
where $\bar{\mu} : H \to H$ is given by $\bar{\mu} := \mu \circ \Pi$, and $\bar{\sigma} : H \to L_2(U_0,H)$ is given by $\bar{\sigma} := \sigma \circ \Pi$. Then the functions $\bar{\mu}$ and $\bar{\sigma}$ are continuous and satisfy the linear growth condition. Furthermore, we have $\bar{\mu}(H_0) \subset H_0$ and $\bar{\sigma}(H_0) \subset L_2(U_0,H_0)$, and the mappings $\bar{\mu}|_{H_0} : H_0 \to H_0$ and $\bar{\sigma}|_{H_0} : H_0 \to L_2(U_0,H_0)$ satisfy the linear growth condition with respect to $\| \cdot \|_{H_0}$. Now, let $\nu$ be a probability measure on $(\calx_0,\ccB(\calx_0))$. Let $x_0 : \Omega \to \calx_0$ be a $\calf_0$-measurable random variable with $\bbp \circ x_0 = \nu$, and let $W$ be a $U$-valued Wiener process with covariance operator $\Sigma_W$, defined on some stochastic basis. Note that $x_0$ is also $\calf_0 / \ccB(H)_{\calx_0}$-measurable, because we have $\ccB(\calx_0) = \ccB(H)_{\calx_0}$ by (\ref{Borel-fields-Hilbert}). In order to obtain a weak solution $(X,W)$ to the SDE (\ref{SDE-Hilbert-full}), we proceed as in the proof of \cite[Thm. 2]{GMR} (see also \cite[Thm. 3.12]{Atma-book}), where only deterministic starting points are considered. Note that, apart from the initial conditions, our framework is a special case of that considered in \cite[Thm. 2]{GMR}, because for each $T \in \bbr_+$ the mapping
\begin{align*}
C([0,T];H) \times [0,T] \to H, \quad (f,t) \mapsto f(t)
\end{align*}
is continuous, and hence conditions (B1)--(B3) and (A3) appearing in \cite[Thm. 2]{GMR} are fulfilled. It remains to prove that $(X,W)$ is also a weak solution to the original SDE (\ref{SDE-Hilbert}). For this purpose, we will show that the closed convex cone $\calx$ is invariant; more precisely that $X \in \calx$ up to an evanescent set. Let $i \in \bbn$ be arbitrary. For each $x \in H$ with $\la x,e_i \ra_H \leq 0$ we have $\la \Pi(x),e_i \ra_H = 0$, and hence, by Propositions \ref{prop-inward} and \ref{prop-parallel} we obtain
\begin{align}\label{inv-cond-proof}
\la \mu(\Pi(x)),e_i \ra_H \geq 0 \quad \text{and} \quad \la \sigma(\Pi(x)),e_i \ra_H = 0.
\end{align}
We define the stopping time
\begin{align*}
S := \inf \{ t \in \bbr_+ : \la X_t, e_i \ra_H < 0 \},
\end{align*}
and claim that $\bbp(S = \infty) = 1$. Suppose, on the contrary, that $\bbp(S < \infty) > 0$. For each $n \in \bbn$ we define the stopping time 
\begin{align*}
T_n := \inf \bigg\{ t \in \bbr_+ : \la X_t, e_i \ra_H < -\frac{1}{n} \bigg\}.
\end{align*}
By the continuity of sample paths of $X$, there exists $n \in \bbn$ such that $\bbp(B_n) > 0$, where  
\begin{align*}
B_n := \{ T_n < \infty \} \cap \{ \la X_t, e_i \ra_H \leq 0 \text{ for all } t \in [\![ S,T_n ]\!] \}.
\end{align*}
Note that $B_n \in \calf$. Indeed, by \cite[Props. I.1.21 and I.1.23]{JacodShiryaev} we have 
\begin{align*}
[\![ S,T_n ]\!] \in \mathscr{O} \subset \calf \otimes \ccB(\bbr_+),
\end{align*}
and hence
\begin{align*}
B_n = \{ T_n < \infty \} \cap \bigcap_{t \in \bbq_+} \Big( \{ \la X_t, e_i \ra_H \leq 0 \} \cap [\![ S,T_n ]\!]_{(\bullet,t)} \Big) \in \calf.
\end{align*}
On the set $B_n$ we have 
\begin{align*}
S < T_n < \infty \quad \text{as well as} \quad \la X_{S},e_i \ra_H = 0 \quad \text{and} \quad \la X_{T_n},e_i \ra_H = -\frac{1}{n}. 
\end{align*}
Furthermore, we have
\begin{align*}
\la X,e_i \ra_H \leq 0 \quad \text{on $[\![ S,T_n ]\!] \cap (B_n \times \bbr_+)$.} 
\end{align*}
Therefore, using (\ref{inv-cond-proof}), on the set $B_n$ we obtain
\begin{align*}
-\frac{1}{n} &= \la X_{T_n},e_i \ra_H - \la X_S,e_i \ra_H = \int_S^{T_n} \la \bar{\mu}(X_s),e_i \ra_H \, ds + \int_{S}^{T_n} \la \bar{\sigma}(X_s),e_i \ra_H \, dW_s
\\ &= \int_S^{T_n} \la \mu(\Pi(X_s)),e_i \ra_H \, ds + \int_{S}^{T_n} \la \sigma(\Pi(X_s)),e_i \ra_H \, dW_s \geq 0,
\end{align*}
which is a contradiction. Consequently, we have $X \in \calx$ up to an evanescent set, and hence $(X,W)$ is also a weak solution to the SDE (\ref{SDE-Hilbert}).
\end{proof}

We conclude this appendix with sufficient conditions for pathwise uniqueness. The following result is a version of \cite[Thm. 2]{Yamada-Watanabe-1971} in infinite dimension.

\begin{theorem}\label{thm-SDE-uniqueness}
We suppose that the following conditions are fulfilled:
\begin{enumerate}[(i)]
\item We have $U = H$, and the operator $\Sigma_W$ has a diagonal structure along the orthonormal basis $( e_k )_{k \in \bbn}$.

\item There exists a sequence $L = (L_i)_{i \in \bbn} \subset \bbr_+$ with $L \in \ell^1(\bbn)$ such that for each $i \in \bbn$ we have
\begin{align}\label{mu-pathwise}
| \mu_i(x) - \mu_i(y) | \leq L_i \|x-y\|_H \quad \text{for all $x,y \in \calx$,}
\end{align}
where $\mu_i : H \to \bbr$ is defined as
\begin{align*}
\mu_i(x) := \la \mu(x),e_i \ra_H, \quad x \in H.
\end{align*}

\item For each $i \in \bbn$ there is a mapping $\sigma_i : \bbr_+ \to \bbr_+$ with $\sigma_i(0) = 0$ such that
\begin{align}\label{sigma-pathwise-structure}
\sigma(x)u = \sum_{i \in \bbn} \sigma_i( \la x,e_i \ra_H) \la g_i,u \ra_{U_0} \, e_i \quad \text{for all $x \in \calx$ and $u \in U_0$,}
\end{align}
where $g_i = \Sigma_W^{\nicefrac{1}{2}} e_i$ for all $i \in \bbn$.

\item There exists a measurable, increasing function $\rho : \bbr_+ \to \bbr_+$ with $\rho(0) = 0$ and $\rho(u) \in (0,\infty)$ for all $u \in (0,\infty)$ satisfying
\begin{align}\label{int-diverges}
\int_0^{\epsilon} \frac{1}{\rho(u)^2} du = \infty \quad \text{for all $\epsilon > 0$}
\end{align}
such that for each $i \in \bbn$ we have
\begin{align}\label{sigma-pathwise}
|\sigma_i(x) - \sigma_i(y)| \leq \rho(|x-y|) \quad \text{for all $x,y \in \bbr_+$.}
\end{align}
\end{enumerate}
Then we have pathwise uniqueness with starting points in $\calx$ for the SDE (\ref{SDE-Hilbert}).
\end{theorem}

\begin{proof}
The proof is similar to that of \cite[Thm. 1]{Yamada-Watanabe-1971}, and we only sketch the most relevant arguments. The sequence $(\lambda_k)_{k \in \bbn} \subset (0,\infty)$ given by $\lambda_k := \la \Sigma_W e_k,e_k \ra_U$ for $k \in \bbn$ satisfies $\sum_{k \in \bbn} \lambda_k < \infty$ and we have $\Sigma_W e_k = \lambda_k e_k$ for all $k \in \bbn$. Furthermore, the system $(g_i)_{i \in \bbn}$ is an orthonormal basis of $U_0$. Let $(X,W)$ and $(Y,W)$ be two weak solutions to the SDE (\ref{SDE-Hilbert}) with $\bbp(X_0 = Y_0) = 1$. We set $Z := X - Y$. Let $i \in \bbn$ be arbitrary. By the diagonal structure (\ref{sigma-pathwise-structure}) of $\sigma$ we have $\bbp$-almost surely
\begin{align*}
\la Z_t,e_i \ra_H = \int_0^t \big( \mu_i(X_s) - \mu_i(Y_s) \big) ds + \int_0^t \big( \sigma_i(\la X_s,e_i \ra_H) - \sigma_i(\la Y_s,e_i \ra_H) \big) d \beta_s^i , \quad t \in \bbr_+,
\end{align*}
where the process
\begin{align*}
\beta^i := \frac{1}{\sqrt{\lambda_i}} \la W,e_i \ra_U 
\end{align*}
is a real-valued standard Wiener process; see \cite[Prop. 4.3.ii]{DaPratoZabczyk}. Using (\ref{int-diverges}), we can choose a sequence $(\varphi_n)_{n \in \bbn}$ of functions $\varphi_n \in C^2(\bbr)$ precisely as in the proof of \cite[Thm. 1]{Yamada-Watanabe-1971}. Now, let $t \in \bbr_+$ and $n \in \bbn$ be arbitrary. By It\^{o}'s formula as well as (\ref{mu-pathwise}) and (\ref{sigma-pathwise}) we obtain
\begin{align*}
\bbe \big[ \varphi_n(\la Z_t,e_i \ra_H) \big] \leq L_i \int_0^t \bbe \big[ \| Z_s \|_H \big] ds + \frac{t}{n}.
\end{align*}
Therefore, letting $n \to \infty$ we obtain
\begin{align*}
\bbe \big[ | \la Z_t,e_i \ra_H | \big] \leq L_i \int_0^t \bbe \big[ \| Z_s \|_H \big] ds.
\end{align*}
Using the monotone convergence theorem, we deduce
\begin{align*}
\bbe \big[ \| Z_t \|_H \big] &\leq \bbe \bigg[ \sum_{i \in \bbn} | \la Z_t,e_i \ra_H | \bigg] = \sum_{i \in \bbn} \bbe \big[ | \la Z_t,e_i \ra_H | \big] \leq \| L \|_{\ell^1(\bbn)} \int_0^t \bbe \big[ \| Z_s \|_H \big] ds.
\end{align*}
Since $L \in \ell^1(\bbn)$, by Gronwall's inequality we obtain $X = Y$ up to indistinguishability, which concludes the proof.
\end{proof}

\section{Linear operators in Hilbert spaces}\label{app-linear-operators}

In this appendix we provide the required results about linear operators in Hilbert spaces. 

\begin{proposition}\cite[Satz VI.3.6]{Werner-2007}
Let $H_0$ and $H$ be separable Hilbert spaces. For every compact linear operator $T \in K(H_0,H)$ there exist orthonormal bases $( f_k )_{k \in \bbn}$ of $H_0$ and $( e_k )_{k \in \bbn}$ of $H$, and a decreasing sequence $(s_k)_{k \in \bbn} \subset \bbr_+$ with $s_k \downarrow 0$ such that
\begin{align}\label{compact-repr}
Tx = \sum_{k=1}^{\infty} s_k \la x,f_k \ra_{H_0} \, e_k \quad \text{for each $x \in H_0$.} 
\end{align}
\end{proposition}

The numbers $(s_k^2)_{k \in \bbn}$ are the eigenvalues of $T^* T$, and the $(s_k)_{k \in \bbn}$ are called the \emph{singular values} of $T$. We say that a compact linear operator $T \in K(H_0,H)$ with representation (\ref{compact-repr}) has \emph{positive singular values} if $s_k > 0$ for all $k \in \bbn$. For what follows, let $H$ be a separable Hilbert space.

\begin{lemma}\label{lemma-zusammenziehen}
Let $T \in K^{++}(H)$ be a compact, self-adjoint, strictly positive linear operator with representation
\begin{align}\label{T-repr}
Tx = \sum_{k \in \bbn} \lambda_k \la x,e_k \ra_H \, e_k \quad \text{for each $x \in H$,}
\end{align}
where $(e_k)_{k \in \bbn}$ is an an orthonormal basis of $H$, and $(\lambda_k)_{k \in \bbn} \subset (0,\infty)$ is a decreasing sequence with $\lambda_k \downarrow 0$. Then the following statements are true:
\begin{enumerate}[(i)]
\item The space $H_0 := T(H)$ equipped with the inner product
\begin{align}\label{inner-prod-T}
\la x,y \ra_{H_0} := \la T^{-1}x, T^{-1}y \ra_H, \quad x,y \in H_0
\end{align}
is a separable Hilbert space, which is dense in $H$.

\item The operator $T : (H,\| \cdot \|_{H}) \to (H_0,\| \cdot \|_{H_0})$ is an isometric isomorphism.

\item The system $( f_k )_{k \in \bbn}$ given by
\begin{align*}
f_k := T e_k = \lambda_k e_k, \quad k \in \bbn
\end{align*}
is an orthonormal basis of $H_0$.

\item For all $x \in H_0$ and $k \in \bbn$ we have
\begin{align}\label{inner-prod-ek}
\la x,e_k \ra_H = \lambda_k \la x,f_k \ra_{H_0}.
\end{align}

\item We have the representation
\begin{align}\label{repr-H0}
H_0 = \bigg\{ x \in H : \sum_{k \in \bbn} \frac{1}{\lambda_k^2} | \la x,e_k \ra_H |^2 < \infty \bigg\}.
\end{align}

\item The identity operator $\Id : (H_0,\| \cdot \|_{H_0}) \to (H,\| \cdot \|_H)$ is a compact linear operator with positive singular values.

\item We have $H_0 \in \ccB(H)$ and $\ccB(H_0) = \ccB(H)_{H_0}$.
\end{enumerate}
\end{lemma}

\begin{proof}
All statements are straightforward to check. For (vii) note that $\ccB(H)_{H_0} \subset \ccB(H_0)$ by the continuity of the identity operator in (vi), and hence by Kuratowski's theorem (see, for example \cite[Thm. I.3.9]{Parthasarathy}) we obtain $H_0 \in \ccB(H)$ and $\ccB(H_0) = \ccB(H)_{H_0}$.
\end{proof}

Informally, we will call a space of the form $H_0 = T(H)$, as provided in Lemma \ref{lemma-zusammenziehen}, a retracted subspace with compact embedding.

\begin{lemma}\label{lemma-function-commute}
Let $T \in K^{++}(H)$ be a compact linear operator as in Lemma \ref{lemma-zusammenziehen}, and let $f : H \to H$ be a mapping such that $fT = Tf$. Then the following statements are true:
\begin{enumerate}[(i)]
\item We have $f(H_0) \subset H_0$.

\item If $f$ is continuous with respect to $\| \cdot \|_H$, then $f|_{H_0} : H_0 \to H_0$ is continuous with respect to $\| \cdot \|_{H_0}$.

\item If $f$ satisfies the linear growth condition with respect to $\| \cdot \|_H$, then $f|_{H_0} : H_0 \to H_0$ satisfies the linear growth condition with respect to $\| \cdot \|_{H_0}$.
\end{enumerate}
\end{lemma}

\begin{proof}
Recall that $H_0 = T(H)$. Since $fT = Tf$, we have
\begin{align*}
f|_{H_0} = f T T^{-1} = T f T^{-1},
\end{align*}
showing that $f(H_0) \subset H_0$. By Lemma \ref{lemma-zusammenziehen} the operator $T : (H,\| \cdot \|_{H}) \to (H_0,\| \cdot \|_{H_0})$ is an isometric isomorphism, proving the remaining statements.
\end{proof}

\end{appendix}

\bibliographystyle{tsagsm}
\bibliography{thorsten}

\end{document}